\documentclass[12pt,a4paper]{article}
\usepackage{latexsym,amssymb,amsfonts,amsmath,amsthm,nccmath,float,enumitem,setspace,bm,graphicx,multirow,mathtools}
\usepackage{amssymb}
\usepackage{amsfonts}
\usepackage{graphicx}
\usepackage{amsmath}
\usepackage{enumitem}
\usepackage[dvipsnames]{xcolor}
\usepackage{cellspace}
\usepackage{nccmath}
\usepackage[hidelinks]{hyperref}
\usepackage[margin=2cm]{geometry}
\usepackage[square,numbers]{natbib}

\setlist{topsep=3pt,partopsep=0pt,itemsep=1pt,parsep=0pt}

\hypersetup{
    colorlinks,
    linkcolor={red!80!black},
    citecolor={blue!80!black},
    urlcolor={blue!80!black}
}

\newcommand{\floor}[1]{\left\lfloor #1 \right\rfloor}
\newcommand{\ceil}[1]{\left\lceil #1 \right\rceil}
\newcommand{\tceil}[1]{\lceil #1 \rceil}
\newcommand{\tfloor}[1]{\lfloor #1 \rfloor}
\renewcommand{\mod}[1]{\pmod{#1}}
\newtheorem{theorem}{Theorem}

\newtheorem{conjecture}{Conjecture}
\newtheorem{lemma}{Lemma}[section]
\newtheorem{definition}[lemma]{Definition}

\def \E {\mathcal{E}}
\def \A {\mathcal{A}}
\def \B {\mathcal{B}}

\def \F {\mathcal{F}}

\def \SP {{\rm SP}}
\def \L {{\rm LL}}
\def \erf {{\rm erf}}
\def \e {\varepsilon}
\def \MMS {{\rm MMS}}

\def \leq {\leqslant}
\def \geq {\geqslant}

\let\oldproofname=\proofname
\renewcommand{\proofname}{\rm\bf{\oldproofname}}

\title{More constructions for Sperner partition systems}
\author{Adam Gowty \qquad  Daniel Horsley\\[2mm]
School of Mathematics, Monash University, Victoria 3800, Australia}
\date{}

\begin{document}
\maketitle
\setstretch{1.2}
\begin{abstract}
An $(n,k)$-\emph{Sperner partition system} is a set of partitions of some $n$-set such that each partition has $k$ nonempty parts and no part in any partition is a subset of a part in a different partition. The maximum number of partitions in an $(n,k)$-Sperner partition system is denoted $\SP(n,k)$. In this paper we introduce a new construction for Sperner partition systems based on a division of the ground set into many equal-sized parts. We use this to asymptotically determine $\SP(n,k)$ in many cases where $\frac{n}{k}$ is bounded as $n$ becomes large. Further, we show that this construction produces a Sperner partition system of maximum size for numerous small parameter sets $(n,k)$. By extending a separate existing construction, we also establish the asymptotics of $\SP(n,k)$ when $n \equiv k \pm 1 \mod{2k}$ for almost all odd values of $k$.
\end{abstract}

\section{Introduction}
A \emph{Sperner set system} is a collection of subsets of some ground set where no set in the collection is a subset of another. Sperner partition systems are a natural variant of Sperner set systems which were first introduced by Meagher, Moura, and Stevens in \cite{MeaMouSte}. An $(n,k)$-\emph{Sperner partition system} is a set of partitions of some $n$-set such that each partition has $k$ nonempty parts and no part in any partition is a subset of a part in a different partition. As with many similar problems, research related to Sperner partition systems focuses on studying the size of the largest possible system for a given choice of $n$ and $k$, which is denoted $\SP(n,k)$.

For a pair of integers $(n,k)$ with $n \geq k \geq 1$, we define $c=c(n,k)$ and $r=r(n,k)$ as the unique integers such that $n=ck+r$ and $r \in \{0,\ldots,k-1\}$. These definitions for $c(n,k)$ and $r(n,k)$ are used throughout the paper and are abbreviated to simply $c$ and $r$ where there is no danger of confusion. In \cite{MeaMouSte} the authors were able to show using the LYM inequality (see \cite[p. 25]{FraTok} for example)
that $\SP(n,k)\leq \MMS(n,k)$, where
\[\MMS(n,k) = \frac{\binom{n}{c}}{k-r+\frac{r(c+1)}{n-c}}.\]
It is also observed in \cite{MeaMouSte} that this upper bound, together with Baranyai's theorem \cite{baranyaifactorization},  implies that $\SP(n,k) =\MMS(n,k)= \binom{n-1}{c-1}$ when $n=ck$. Furthermore, it is trivial to see $\SP(n,k) = 1$ when $k=1$ and in  \cite{LiMeagher} it is proved that $\SP(n,2) = \binom{n-2}{\floor{n/2}-1}$, so only cases with $k \geq 3$ remain unresolved.

In many cases where $c$ grows along with $n$, Chang et. al. \cite{Paper1} constructed Sperner partition systems with number of partitions asymptotic to $\MMS(n,k)$.

\begin{theorem}[\cite{Paper1}]\label{t:asymptotic}
Let $n$ and $k$ be integers with $n \rightarrow \infty$, $k=o(n)$ and $k \geq 3$, and let $c$ and $r$ be the integers such that $n=ck+r$ and $r \in \{0,\ldots,k-1\}$. Then $\SP(n,k)\sim\MMS(n,k)$ if
\begin{itemize}
    \item
$n$ is even and $r \notin \{1,k-1\}$; or
    \item
$k-r \rightarrow \infty$.
\end{itemize}
\end{theorem}

The condition $k=o(n)$ in Theorem~\ref{t:asymptotic} is equivalent to saying $c \rightarrow \infty$, so Theorem~\ref{t:asymptotic} does not cover the case where $c$ is bounded as $n$ grows. Certain very specific cases in the regime where $c$ is bounded have been investigated. When $c=1$, it is not hard to see that $\SP(n,k)=1$. For $c=2$, Li and Meagher \cite{LiMeagher} found bounds on $\SP(2k+1,k)$, $\SP(2k+2,k)$ and $\SP(3k-1,k)$, and in \cite{Paper1} it was shown that $\SP(3k-6,k)=\lfloor\frac{1}{2}(k-2)^2\rfloor$ for $k \geq 11$ and $k \not\equiv 4 \mod{6}$. Our first main contribution in this paper gives an  asymptotic determination of $\SP(n,k)$ when $c$ is bounded, except in cases where $r$ is very close to $k$.

\begin{theorem}\label{t:main}
Let $n$ and $k$ be integers with $n \rightarrow \infty$, $k \leq \frac{n}{2}$ and $k-r=\Theta(n)$ where $c$ and $r$ are the integers such that $n=ck+r$ and $r \in \{0,\ldots,k-1\}$. Then $\SP(n,k)\sim\MMS(n,k)$.
\end{theorem}

Note that the condition $k-r=\Theta(n)$ implies $k=\Theta(n)$ and hence that $c=O(1)$. Theorem~\ref{t:main} is proved by introducing a new construction for Sperner partition systems which is based on a division of the ground set into many equal-sized parts (see Lemma~\ref{l:mainConstruction}). In the special case $c=2$ we are able to say more (see Section~\ref{S:c=2}), including determining $\SP(n,k)$ exactly for a number of small parameter sets $(n,k)$ and narrowing it down to one of two values for one infinite family.

\begin{theorem}\label{T:3ktake2case}
Let $k\geq 4$ be an even integer and let $n = 3k - 2$. Then $\SP(n,k) \in \{\binom{n/2}{2}, \binom{n/2}{2}+1\}$.
\end{theorem}

Theorem~\ref{t:asymptotic} also does not cover the cases where $r=1$ and $k$ is bounded or where $r=k-1$. Here we show that in most of these cases, if $n$ is even, $\SP(n,k)$ is indeed asymptotic to $\MMS(n,k)$.

\begin{theorem}\label{t:k=3Asymptotics}{~}
\begin{itemize}
    \item[\textup{(a)}]
Let $n$ and $k$ be integers such that $n \rightarrow \infty$ with $n \equiv k+1 \mod{2k}$, $k=o(n)$, and $k \geq 3$ is odd. Then $\SP(n,k) \sim \MMS(n,k)$.
    \item[\textup{(b)}]
Let $n$ and $k$ be integers such that $n \rightarrow \infty$ with $n \equiv k-1 \mod{2k}$, $k=o(n)$, and $k \geq 5$ is odd. Then $\SP(n,k) \sim \MMS(n,k)$.
\end{itemize}
\end{theorem}

We prove Theorem~\ref{t:k=3Asymptotics} by extending a construction method used in \cite{Paper1} and analysing its behaviour. This extended construction method incorporates a solution to a particular integer program, where the objective value of the program gives the size of the Sperner partition system produced. With some effort, we are able to show that in most cases the optimal value of this integer program is asymptotic to $\MMS(n,k)$ and so prove Theorem~\ref{t:k=3Asymptotics}. In the case $n \equiv k-1 \mod{2k}$ and $k = 3$ we do not prove this, but we present strong numerical evidence that the optimal value of the integer program is such that the result still holds.

Sperner partition systems are of particular interest due to their relationship with detecting arrays (see \cite{Detecting} for example), which have extensive uses in the testing and location of faults. An $(n,k)$-Sperner partition system with $p$ partitions can be represented by an $n\times p$ array with $k$ symbols, where the $(i,j)$ entry is $\ell$ if and only if the $i$th element of the ground set is in the $\ell$th part of the $j$th partition under arbitrary orderings. In the language of \cite{Detecting}, such an array is a $(1,\Bar{1})$-\emph{detecting array} due to the fact that for any $j_1,j_2 \in \{1,\ldots,p\}$ and $\ell_1,\ell_2 \in \{1,\ldots,k\}$, the set of rows for which the symbol $\ell_1$ appears in column $j_1$ is not a subset of the set of rows for which the symbol $\ell_2$ appears in column $j_2$. As such, it is apparent that $\SP(n,k)$ also denotes the maximum number of columns in a $(1,\Bar{1})$-detecting array with exactly $n$ rows and $k$ symbols.

This paper is organised as follows. Section~\ref{S:prelim} introduces some notation we will require as well as a key result, a consequence of a result of \cite{Darryn}, that underlies our constructions. In Section~\ref{S:smallc} we introduce our new construction for Sperner partition systems, based on a division of the ground set into many equal-sized parts, and use this to prove Theorem~\ref{t:main}. In Section~\ref{S:c=2} we then examine the special case where $c=2$, in the process proving Theorem~\ref{T:3ktake2case} and exhibiting many small parameter sets for which the construction from the previous section produces Sperner partition systems of maximum size. Sections~\ref{S:th4a} and~\ref{S:th4b} are then devoted to proving Theorem~\ref{t:k=3Asymptotics}(a) and (b) respectively, using an extension of the construction for Sperner partition systems given in \cite[Lemmas~10 and 11]{Paper1}. In the conclusion, we provide some numerical evidence that Theorem~\ref{t:k=3Asymptotics}(b) also holds for $k=3$ and discuss possible further work.

\section{Preliminaries}\label{S:prelim}

An $(n,k)$-Sperner partition system is called \emph{almost uniform} if each part of each partition in the system is of size $\lfloor{\frac{n}{k}}\rfloor$ or $\lceil{\frac{n}{k}}\rceil$. Note that this means that there must be $k-r$ parts of size $c$ and $r$ parts of size $c+1$ in each partition. It is conjectured in \cite{MeaMouSte} that, for all $n$ and $k$ with $n \geq k \geq 1$, there is an almost uniform Sperner partition system with $\SP(n,k)$ partitions.

In \cite{LiMeagher}, the authors observe that taking an $(n,k)$-Sperner partition system and adding a new element to an arbitrary part of each partition results in an $(n+1,k)$-Sperner partition system of the same size. Thus we have
\begin{equation}\label{e:mono}
\SP(n+1,k) \geq \SP(n,k) \quad \text{for all integers $n \geq k \geq 1$},
\end{equation}
a fact that we will use frequently. If the original Sperner partition system is almost uniform and the new element is added to a part of minimum size in each partition, then the resulting $(n+1,k)$-Sperner partition system is also almost uniform. Although we do not state it explicitly each time, all the constructions in this paper produce almost uniform systems. For a set $S$ and a nonnegative integer $i$, we denote the set of all $i$-subsets of $S$ by $\binom{S}{i}$. Note that $\left|\binom{S}{i}\right| = \binom{|S|}{i}$.

A \emph{hypergraph} $H$ consists of a vertex set $V(H)$ together with a set $\E(H)$ of edges, each of which is a nonempty subset of $V(H)$. We do not allow multiple edges. A \emph{clutter} is a hypergraph for which no edge is a subset of another. As such, a clutter is exactly a Sperner set system, but we use the term clutter when we wish to consider the object through a hypergraph-theoretic lens.

In this paper, a \emph{partial edge colouring} of a hypergraph is simply an assignment of colours to some subset of its edges with no further conditions imposed. Let $\gamma$ be a partial edge colouring of a hypergraph $H$ with colour set $C$. For each $z \in C$, the set $\gamma^{-1}(z)$ of edges of $H$ assigned colour $z$ is called a \emph{colour class} of $\gamma$. For each $z \in C$ and $x \in V(H)$, let the number of edges of $H$ that are assigned the colour $z$ by $\gamma$ and contain the vertex $x$ be denoted $\deg^{\gamma}_z(x)$.

Throughout this paper, we will make extensive use of the following result from \cite{Paper1}, which is a consequence of a more general and powerful result of Bryant \cite{Darryn}. It allows the construction of a Sperner partition system to be reduced to finding a partial edge colouring of a hypergraph with appropriate properties, which can greatly simplify the task.

\begin{lemma}[\cite{Paper1}]\label{l:colouring}
Let $n$ and $k$ be integers with $n \geq k \geq 1$, let $H$ be a clutter with $|V(H)|=n$, and let $\{X_1,\ldots,X_t\}$ be a partition of $V(H)$ such that any permutation of $X_w$ is an automorphism of $H$ for each $w \in \{1,\ldots,t\}$. Suppose there is a partial edge colouring $\gamma$ of $H$ with colour set $C$  such that, for each $z \in C$, $|\gamma^{-1}(z)|=k$ and $\sum_{x \in X_w}\deg^{\gamma}_z(x)=|X_w|$ for each $w \in \{1,\ldots,t\}$. Then there is an $(n,k)$-Sperner partition system with $|C|$ partitions such that the parts of the partitions form a subset of $\E(H)$.
\end{lemma}

\section{Proof of Theorem~\ref{t:main}}\label{S:smallc}

Our goal in this section is to prove Theorem~\ref{t:main}. We achieve this by first introducing a new construction for Sperner partition systems and then showing that the construction can produce systems with size asymptotic to $\MMS(n,k)$ in the regime where $c$ is bounded and $r$ is not too close to $k$.

We now introduce a simple lemma which will be useful in detailing our construction. It will eventually allow us to distribute the edges of a hypergraph evenly between colour classes when attempting to define a colouring satisfying the hypotheses of Lemma~\ref{l:colouring}.

\begin{lemma}\label{l:balancedmatrix}
Let $s_1$ and $s_2$ be positive integers and $x$ and $a \leq s_2$ be nonnegative integers. There exists an $s_1 \times s_2$ matrix $T$ such that each row of $T$ has $a$ occurrences of $x+1$ and $s_2-a$ occurrences of $x$, and the sums of any two columns of $T$ differ by at most $1$.
\end{lemma}
\begin{proof}
We proceed by induction on $s_1$. The result is clearly true when $s_1=1$, so let $s'_1 \in \{1,\ldots,s_1-1\}$ and suppose there exists an $s'_1 \times s_2$ matrix $T'$ with the required properties. Let $Y$ be the set of columns of $T'$ whose sum is the minimum column sum in $T'$. Add to $T'$ a new row with $a$ occurrences of $x+1$ and $s_2-a$ occurrences of $x$, placed so that each column in $Y$ contains an occurrence of $x+1$ if $a \geq |Y|$ and so that each occurrence of $x+1$ is in a column in $Y$ if $a < |Y|$. It can be checked that the resulting matrix has the required properties.
\end{proof}

We now introduce the construction that will be used to prove Theorem~\ref{t:main}.

\begin{lemma}\label{l:mainConstruction}
Let $n$, $c$, $k$ and $r$ be integers such that $n = ck + r$, $c \geq 2$, $k \geq 3$ and $r \in \{1, 2, \ldots, k-1\}$. Suppose that $n=hm$ for positive integers $m$ and $h$ such that $m \equiv 0 \mod{c}$ and let
\[p_1 = \floor{\mfrac{m(h^c-c-1)}{c(k-r)}}, \quad p_2 = \floor{\mfrac{ m}{r}\floor{\mfrac{\binom{h}{c+1}}{\binom{m-1}{c-1}}}}, \quad p'_1 = \floor{\mfrac{mh^c}{c(k-r)}}, \quad p'_2 = \floor{\mfrac{ m\binom{h}{c+1}}{r\binom{m-1}{c-1}}}.\]
There exists an almost uniform $(n,k)$-Sperner partition system with $p\binom{m-1}{c-1}$ partitions if
\begin{itemize}
    \item[\textup{(a)}]
$p=\min\{p_1,p_2\}$; or
    \item[\textup{(b)}]
$p=\min\{p'_1,p'_2\}$ and $pr \equiv 0 \mod{m}$.
\end{itemize}
\end{lemma}

\begin{proof}
Suppose the hypotheses of (a) hold or that those of (b) do. First note that $r \equiv 0 \mod{c}$ since $n \equiv 0 \mod{c}$ and $ck \equiv 0 \mod{c}$. We will construct our Sperner partition system on a ground set $X = X_1 \cup \cdots \cup X_m$, where $X_1,\ldots,X_m$ are pairwise disjoint sets such that $|X_1| =  \cdots = |X_m| = h$. Let $\mathcal{M}=\binom{\{1,\ldots,m\}}{c}$ and let $J=\{1,\ldots,\binom{m-1}{c-1}\} \times \{1,\ldots,\frac{m}{c}\}$. By Baranyai's theorem \cite{baranyaifactorization}, we can index the sets in $\mathcal{M}$ so that $\mathcal{M}=\{S_{\ell,i}:(\ell,i) \in J\}$ and $\{S_{\ell,i}: i \in \{1,\ldots,\frac{m}{c}\}\}$ is a partition of $\{1,\ldots,m\}$ for each $\ell \in \{1,\ldots,\binom{m-1}{c-1}\}$. Let $H$ be a hypergraph with vertex set $X$ and edge set $\A \cup \B$ where
\begin{align*}
\A&=\medop\bigcup_{(\ell,i) \in J} \mathcal{A}_{\ell,i} &&\mbox{for}& \mathcal{A}_{\ell,i} &=\left\{E \in \tbinom{X}{c}: \mbox{$|E \cap X_w|=1$ for each $w \in S_{\ell,i}$} \right\}, \\
\B&=\medop\bigcup_{w=1}^{m} \mathcal{B}_w  &&\mbox{for}& \mathcal{B}_w &= \tbinom{X_w}{c+1}.
\end{align*}
The indexing of the sets in $\mathcal{M}$ will act as a guide for a partial edge colouring of $H$. Let $C'$ be a set with $|C'|=p$ and let $C=C' \times  \{1,2,\ldots,\binom{m-1}{c-1}\}$ be a set of colours. It is clear that any permutation of $X_w$ is an automorphism of $H$ for each $w \in \{1,\ldots,m\}$. Thus, by Lemma~\ref{l:colouring}, to find an $(n,k)$-Sperner partition system with $p\binom{m-1}{c-1}$ partitions, it suffices to find a partial edge colouring $\gamma$ of $H$ with colour set $C$ such that for each $z \in C, |\gamma^{-1}(z)| = k$ and $\sum_{x \in X_w} \deg_z^\gamma(x) = h$ for each $w\in \{1, \ldots, m\}$. We proceed to show that such a partial edge colouring $\gamma$ exists.

Let $x = \floor{\frac{r}{m}}$ and $a = \frac{1}{c}(r-mx)$, noting that $a$ will be a non-negative integer and let $T=(t_{z,i})$ be a $p\times \frac{m}{c}$ matrix with $a$ occurrences of $x+1$ and $\frac{m}{c}-a$ occurrences of $x$ in each row such that the sums of any two columns differ by at most $1$. Such a matrix exists by Lemma~\ref{l:balancedmatrix}. We consider the rows of $T$ to be indexed by the elements of $C'$. It follows from our definition of $a$ that the sum of any row of $T$ is $\frac{r}{c}$. Thus the sum of all the entries in $T$ is $\frac{pr}{c}$ and hence, because $T$ has $\frac{m}{c}$ columns and the sums of any two of these differ by at most 1, the sum of any column of $T$ is in $\{\lfloor\frac{pr}{m}\rfloor,\lceil\frac{pr}{m}\rceil\}$. So we have
\begin{equation}\label{e:matrixLineSums}
\medop\sum_{i=1}^{m/c}t_{z,i}=\tfrac{r}{c} \mbox{ for each $z \in C'$},\quad \mbox{and} \quad \medop\sum_{z \in C'}t_{z,i} \in \left\{\left\lfloor\tfrac{pr}{m}\right\rfloor,\left\lceil\tfrac{pr}{m}\right\rceil\right\} \mbox{ for each $i \in \{1,\ldots,\frac{m}{c}\}$}.
\end{equation}

We create our partial edge colouring $\gamma$ of $H$ by, for all $(z,\ell) \in C$ and $i \in \{1,\ldots,\frac{m}{c}\}$, one at a time in arbitrary order, performing the following process.
\begin{itemize}
    \item
Assign the colour $(z,\ell)$ to $h-(c+1)t_{z,i}$ previously uncoloured edges in $\mathcal{A}_{\ell,i}$.
    \item
For each $w \in S_{\ell,i}$, assign the colour $(z,\ell)$ to $t_{z,i}$ previously uncoloured edges in $\mathcal{B}_w$.
\end{itemize}
After performing this process for all $(z,\ell) \in C$ and $i \in \{1,\ldots,\frac{m}{c}\}$, we call the resulting colouring $\gamma$. We will show that there are always uncoloured edges available throughout this process and that $\gamma$ satisfies the conditions we require of it.
\begin{itemize}
    \item[\textup{(i)}]
Let $(\ell,i) \in J$. We show that the number of edges in $\A_{\ell,i}$ assigned colours is at most $|\A_{\ell,i}|=h^c$. These edges only receive colours in $C' \times \{\ell\}$ and, for each $z \in C'$, the number that receive colour $(z,\ell)$ is $h-(c+1)t_{z,i}$. So the total number that are assigned a colour is $\sum_{z \in C'} (h-(c+1)t_{z,j})$, which is at most $ph - (c+1)\lfloor\frac{pr}{m}\rfloor$ by \eqref{e:matrixLineSums}.
If the hypotheses of (a) hold then
\[ph - (c+1)\left\lfloor\mfrac{pr}{m}\right\rfloor < p h - (c+1)\left(\mfrac{pr}{m}-1\right) = \mfrac{p c(k-r)}{m}+c+1 \leq h^c\]
where the equality follows by substituting $h=\frac{1}{m}(ck+r)$ and the last inequality is obtained using $p \leq p_1$ and the definition of $p_1$. If the hypotheses of (b) hold then similarly we have
\[ph - (c+1)\left\lfloor\mfrac{pr}{m}\right\rfloor= ph - (c+1)\mfrac{pr}{m} = \mfrac{p c(k-r)}{m} \leq h^c
\]
where the last inequality is obtained using $p \leq p'_1$ and the definition of $p'_1$.
    \item[\textup{(ii)}]
Let $w \in \{1,\ldots,m\}$. We show that we do not run out of uncoloured edges in $\mathcal{B}_w$ by showing that, for each $\ell \in \{1,\ldots,\binom{m-1}{c-1}\}$, the number of edges in $\mathcal{B}_w$ assigned a colour in $C' \times \{\ell\}$ is at most $|\mathcal{B}_w|/\binom{m-1}{c-1}=\binom{h}{c+1}/\binom{m-1}{c-1}$. Let $\ell \in \{1,\ldots,\binom{m-1}{c-1}\}$ and let $i$ be the unique element of $\{1,\ldots,\frac{m}{c}\}$ such that $w \in S_{\ell,i}$. Then the number of edges in $\mathcal{B}_w$ assigned a colour in $C' \times \{\ell\}$ is $\sum_{z \in C'}t_{z,i}$, and this is at most $\lceil\frac{pr}{m}\rceil$ by \eqref{e:matrixLineSums}. If the hypotheses of (a) hold then, $\lceil\frac{pr}{m}\rceil \leq  \lceil\frac{p_2r}{m}\rceil$ and we obtain the required bound using the definition of $p_2$. If the hypotheses of (b) hold then, $\lceil\frac{pr}{m}\rceil=\frac{pr}{m} \leq  \frac{p'_2r}{m}$ and we can obtain the required bound using the definition of $p'_2$.
    \item[\textup{(iii)}]
Let $(z,\ell)$ be a colour in $C$. We show that $|\gamma^{-1}((z,\ell))|=k$. For each $i \in \{1,\ldots,\frac{m}{c}\}$, we assign $(z,\ell)$ to $h-(c+1)t_{z,i}$ edges in $\A$ and, because $|S_{\ell,i}|=c$, to $ct_{z,i}$ edges in $\B$. So
\[\left|\gamma^{-1}\bigl((z,\ell)\bigr)\right| = \medop\sum_{i=1}^{m/c}(h-(c+1)t_{z,i}+ct_{z,i}) = \mfrac{hm}{c}-\medop\sum_{i=1}^{m/c}t_{z,i} = \mfrac{hm}{c}-\mfrac{r}{c}=k\]
where the third equality follows by \eqref{e:matrixLineSums} and the last because $hm=ck+r$.
    \item[\textup{(iv)}]
Let $w \in \{1,\ldots,m\}$ and $(z,\ell) \in C$. We show that $\sum_{x \in X_{w}}\deg_{(z,\ell)}^\gamma(x)=h$. Let $i$ be the unique element of $\{1,\ldots,\frac{m}{c}\}$ such that $w \in S_{\ell,i}$. Then $(z,\ell)$ is assigned to $h-(c+1)t_{z,i}$ edges in $\A_{\ell,i}$, each of which contains one vertex in $X_w$, and to $t_{z,i}$ edges in $\B_{w}$, each of which contains $c+1$ vertices in $X_w$. Any other edges assigned $(z,\ell)$ are disjoint from $X_w$. Thus
\[\medop\sum_{x \in X_{w}}\deg_{(z,\ell)}^\gamma(x)=(h-(c+1)t_{z,j})+(c+1)t_{z,j}=h.\]
\end{itemize}
So by (i) and (ii) we can indeed obtain the partial edge colouring $\gamma$ as we claimed and by (iii) and (iv) $\gamma$ has the required properties. So we can apply Lemma~\ref{l:colouring} to obtain an almost uniform $(n,k)$-Sperner partition system with $p\binom{m-1}{c-1}$ partitions as discussed.
\end{proof}

Note that we could potentially include all $c$-subsets of $X$ that are not subsets of an $X_i$ as edges of our clutter above. However, attempting to use all of these would make finding a suitable colouring $\gamma$ more difficult. Moreover, as $m$ becomes large, the number of $c$-sets we do not use is an asymptotically insignificant fraction of the number of those that we do. With this new construction for Sperner partition systems in hand, we are now able to prove Theorem~\ref{t:main}.


\begin{proof}[\textup{\textbf{Proof of Theorem~\ref{t:main}.}}]
Observe that we have $n=\Theta(k)=\Theta(k-r)$ and hence $c=O(1)$. We consider two cases according to the value of $r$.

\textbf{Case 1.} Suppose that $r \leq k^{(2c-1)/2c}$. So $r=o(k)$ and $n \sim ck$. Clearly $\SP(ck,k) \leq \SP(n,k) \leq \MMS(n,k)$ using \eqref{e:mono}, so to complete the proof it suffices to show that $\MMS(n,k) \sim \SP(ck,k)$. Note that
\[\MMS(n,k) = \mfrac{\binom{n}{c}}{k-r + \frac{r(c+1)}{n-c}} \sim \mfrac{1}{k}\mbinom{n}{c} \sim \mfrac{n^c}{c!\,k} \sim \mfrac{n^{c-1}}{(c-1)!} \sim \mfrac{(ck)^{c-1}}{(c-1)!} \sim \mbinom{ck}{c-1} = \SP(ck,k)\]
where the first $\sim$ follows because $r=o(k)$, we use $n \sim ck$ frequently throughout and the final equality comes from \cite[Theorem 1]{MeaMouSte} as discussed in the introduction.


\textbf{Case 2.} Suppose that $r > k^{(2c-1)/2c}$. Let
\[h = \left\lceil\left(\mfrac{(c+1)rn^{c-1}}{k-r}\right)^{1/c}\right\rceil \qquad\text{and} \qquad m = \floor{\mfrac{n}{h}}-\delta,\]
where $\delta \in \{0,1, \ldots, c-1\}$ is chosen such that $m \equiv 0 \mod c$. Since $k^{(2c-1)/2c}<r<k$ and $k-r=\Theta(k)=\Theta(n)$, we have $h=O(n^{(c-1)/c})$ but $h=\Omega(n^{(2c^2-2c-1)/2c^2})$. So $mh \leq n$ and $mh = n-O(n^{(c-1)/c})=ck+r-o(r)$. Let $q= mh-ck$ and note that $q \leq r$ and $q = r - o(r)$.

Using $mh \leq n$ and \eqref{e:mono}, we have $\SP(mh,k) \leq \SP(n,k) \leq \MMS(n,k)$. We will complete the proof by showing that $\SP(mh,k) \geq \MMS(n,k)(1+o(1))$. We will use Lemma~\ref{l:mainConstruction}(a) to obtain this lower bound on $\SP(mh,k)$. Let $p_1$ and $p_2$ be as defined in the Lemma~\ref{l:mainConstruction} statement, except with $q$ in place of $r$ (noting that $mh=ck+q$). Now,
\begin{equation}\label{e:alphamms1}
p_2  \sim \mfrac{m\binom{h}{c+1}}{q\binom{m-1}{c-1}} \sim \mfrac{ mh^{c+1}}{qc(c+1)m^{c-1}} \sim \mfrac{ mh^{2c}}{qc(c+1)n^{c-1}} \sim \mfrac{mh^c}{c(k-r)} \sim \mfrac{mh^c}{c(k-q)} \sim p_1
\end{equation}
where the first $\sim$ holds because $\binom{m-1}{c-1}=o(\binom{h}{c+1})$ since $h=\Omega(n^{(2c^2-2c-1)/2c^2})$, the third holds because $m \sim \frac{n}{h}$, the fourth holds by applying the definition of $h$ and then using $q \sim r$, and the fifth holds using $k-r \sim k-q$. By Lemma~\ref{l:mainConstruction}(a) and \eqref{e:alphamms1}, we have
\[\SP(mh,k) \geq \mbinom{m-1}{c-1}\min\{p_1,p_2\} \sim \mbinom{m-1}{c-1}\mfrac{ mh^c}{c(k-r)} \sim \mfrac{ m^ch^c}{c!(k-r)} \sim \mfrac{ n^c}{c!(k-r)} \sim \MMS(n,k)\]
where the first $\sim$ uses \eqref{e:alphamms1}, the third uses $mh \sim n$ and the last uses the definition of $\MMS(n,k)$ together with $k-r=\Theta(k)$ and $\frac{r(c+1)}{n-c}=O(1)$.
\end{proof}

\section{The case \texorpdfstring{$\boldsymbol{c=2}$}{c=2}}\label{S:c=2}

Recall from the introduction that $\SP(n,k)=1$ when $c=1$. Thus, of the cases where $c$ is constant, the first nontrivial case of $c=2$ is of particular interest. Here we first observe two consequences of an upper bound on $\SP(n,k)$ given in \cite{Paper1} which slightly improves on $\MMS(n,k)$. We extend the usual binomial coefficient notation by defining $\binom{q}{t}=\frac{1}{t!}\prod_{i=0}^{t-1}(q-i)$, for any real number $q$ and integer $t$ with $q \geq t \geq 0$.

\begin{theorem}[\cite{Paper1}]\label{t:upperBound}
If $n$ and $k$ are integers such that $n \geq 2k+2$ and $k \geq 4$, then
\begin{equation}\label{e:upperBound}
\bigl\lceil \bigl(1-\tfrac{r(c+1)}{n}\bigr) \cdot \SP(n,k) \bigr\rceil+\L_c\left(\bigl\lfloor \tfrac{r(c+1)}{n}\cdot\SP(n,k) \bigr\rfloor\right) \leq \mbinom{n-1}{c-1},
\end{equation}
where $c$ and $r$ are the integers such that $n=ck+r$ and $r \in \{0,\ldots,k-1\}$ and $\L_c(x) = \binom{q}{c-1}$ with $q$ being the unique nonnegative real number for which $q \geq c$ and $x = \binom{q}{c}$.
\end{theorem}
Note that this acts as an upper bound due to the fact that for fixed nonnegative integers $n$ and $k$, the left hand side of \eqref{e:upperBound} is nondecreasing in $\SP(n,k)$ (see \cite{Paper1} for details). Computation reveals that, in the case $c=2$, there are numerous small parameter sets $(n,k)$ for which we can exactly determine  $\SP(n,k)$ because the upper bound given by Theorem~\ref{t:upperBound} equals the lower bound implied by Lemma~\ref{l:mainConstruction}(b) for some choice of $m$ and $h$ with $mh=n$. In Table~\ref{table:exactValues} (see page~\pageref{table:exactValues}), we list all the parameter sets with $n \leq 1000$ for which this occurs, together with the associated values of $m$, $h$ and $\SP(n,k)$. We have not found any such parameter sets for $c \geq 3$. This is perhaps not surprising because, as discussed immediately after Lemma~\ref{l:mainConstruction}, the construction ``wastes'' some $c$-sets when $c \geq 3$.

\begin{table}[h]
\begin{small}
\begin{center}
\begin{tabular}{rr||rr||r ll rr||rr||r}
$n$ & $k$ & $m$ & $h$ & $\SP(n,k)$ & \rule{1cm}{0cm} & \rule{1cm}{0cm} & $n$ & $k$ & $m$ & $h$ & $\SP(n,k)$ \\ \cline{1-5} \cline{8-12}
36  & 15  & 4   & 9   & 54         &  &  & 560 & 203 & 8   & 70  & 2800       \\
44  & 18  & 4   & 11  & 72         &  &  & 560 & 232 & 16  & 35  & 1080       \\
56  & 22  & 4   & 14  & 117        &  &  & 564 & 220 & 12  & 47  & 1518       \\
88  & 33  & 4   & 22  & 264        &  &  & 576 & 224 & 12  & 48  & 1584       \\
128 & 54  & 8   & 16  & 210        &  &  & 588 & 228 & 12  & 49  & 1650       \\
138 & 54  & 6   & 23  & 330        &  &  & 600 & 224 & 10  & 60  & 2250       \\
144 & 56  & 6   & 24  & 360        &  &  & 600 & 260 & 20  & 30  & 950        \\
144 & 60  & 8   & 18  & 252        &  &  & 624 & 304 & 52  & 12  & 663        \\
150 & 58  & 6   & 25  & 390        &  &  & 640 & 230 & 8   & 80  & 3584       \\
150 & 65  & 10  & 15  & 225        &  &  & 672 & 266 & 14  & 48  & 1664       \\
160 & 66  & 8   & 20  & 294        &  &  & 672 & 273 & 16  & 42  & 1440       \\
168 & 77  & 14  & 12  & 208        &  &  & 680 & 323 & 40  & 17  & 780        \\
230 & 95  & 10  & 23  & 432        &  &  & 700 & 275 & 14  & 50  & 1820       \\
252 & 111 & 14  & 18  & 364        &  &  & 720 & 290 & 16  & 45  & 1620       \\
288 & 105 & 6   & 48  & 1280       &  &  & 720 & 330 & 30  & 24  & 928        \\
288 & 128 & 16  & 18  & 405        &  &  & 750 & 275 & 10  & 75  & 3375       \\
300 & 120 & 10  & 30  & 675        &  &  & 756 & 360 & 42  & 18  & 861        \\
306 & 111 & 6   & 51  & 1445       &  &  & 768 & 352 & 32  & 24  & 992        \\
318 & 115 & 6   & 53  & 1560       &  &  & 770 & 282 & 10  & 77  & 3510       \\
324 & 117 & 6   & 54  & 1620       &  &  & 800 & 335 & 20  & 40  & 1482       \\
330 & 119 & 6   & 55  & 1680       &  &  & 812 & 315 & 14  & 58  & 2301       \\
336 & 144 & 14  & 24  & 546        &  &  & 816 & 289 & 8   & 102 & 5712       \\
336 & 160 & 28  & 12  & 378        &  &  & 840 & 315 & 12  & 70  & 3080       \\
342 & 123 & 6   & 57  & 1805       &  &  & 840 & 350 & 20  & 42  & 1596       \\
360 & 129 & 6   & 60  & 2000       &  &  & 840 & 378 & 30  & 28  & 1160       \\
360 & 135 & 8   & 45  & 1260       &  &  & 852 & 319 & 12  & 71  & 3168       \\
368 & 138 & 8   & 46  & 1288       &  &  & 864 & 342 & 16  & 54  & 2160       \\
378 & 135 & 6   & 63  & 2205       &  &  & 880 & 365 & 20  & 44  & 1710       \\
420 & 175 & 14  & 30  & 780        &  &  & 936 & 348 & 12  & 78  & 3718       \\
480 & 176 & 8   & 60  & 2100       &  &  & 938 & 358 & 14  & 67  & 3003       \\
528 & 192 & 8   & 66  & 2541       &  &  & 944 & 332 & 8   & 118 & 7497       \\
528 & 220 & 16  & 33  & 990        &  &  & 960 & 448 & 40  & 24  & 1170       \\
546 & 221 & 14  & 39  & 1183       &  &  & 994 & 378 & 14  & 71  & 3276
\end{tabular}
\end{center}
\end{small}
\caption{Parameter sets $(n,k)$ for which $\SP(n,k)$ is exactly determined by Theorem~\ref{t:upperBound} and Lemma~\ref{l:mainConstruction}(b), and the associated values of $m$, $h$ and $\SP(n,k)$.}\label{table:exactValues}
\end{table}

In the special case where $c = 2$ and $r$ is small compared to $n$, we are able to give a more explicit form of the bound implied by Theorem~\ref{t:upperBound}.

\begin{lemma}\label{L:smallRUpperBound}
Let $k \geq 4$ and $r$ be integers such that $1 \leq r \leq \frac{1}{3}\sqrt{2k}$ and let $t = \ceil{\L_2(3r)}$. Then $\SP(2k+r,k) \leq 2k+4r-t-1$.
\end{lemma}

\begin{proof}
Suppose for a contradiction that $\SP(n,k) \geq  2k+4r-t$. Then, because the left side of \eqref{e:upperBound} is monotonically increasing in $\SP(n,k)$, Theorem~\ref{t:upperBound} implies that
\begin{equation}\label{E:upperBound}
\ceil{\left(1- \tfrac{3r}{2k + r}\right)\left(2k +4r - t\right)} + \L_2\left(\floor{\tfrac{3r}{2k + r}\left(2k +4r - t\right)}\right) \leq 2k + r-1.
\end{equation}
Observe that $\tfloor{\tfrac{3r}{2k + r}(2k + 4r - t)}=\tfloor{3r + \frac{3r(3r-t)}{2k+r}}=3r$ because $9r^2 \leq 2k$. Using this fact, we have $\tceil{(1- \tfrac{3r}{2k + r})(2k + 4r - t)} =2k + 4r - t - \tfloor{\tfrac{3r}{2k + r}(2k + 4r - t)} = 2k+r-t$. So \eqref{E:upperBound} is equivalent to $2k + r  - t + \L_2(3r) \leq 2k + r - 1$ which is impossible, because  $t-\L_2(3r) < 1$ by the definition of $t$.
\end{proof}
More routine calculations establish that Theorem~\ref{t:upperBound} does not rule out the possibility that $\SP(2k+r,k) = 2k + 4r - 1 - t$.

In \cite{LiMeagher}, it is shown that $2k-1 \leq \SP(2k+1,k) \leq 2k$ and $2k+1 \leq \SP(2k+2,k) \leq 2k+3$. As a consequence of \eqref{e:mono}, the latter means that $2k+1$ is a lower bound for $\SP(2k+r,k)$ for all $r \geq 3$. For small values of $r \geq 3$ we give the upper bound provided by Lemma~\ref{L:smallRUpperBound} on $\SP(2k+r,k)$, together with the range of $k$ values it applies for in Table~\ref{table:upperBounds} below. Lemma~\ref{L:smallRUpperBound} guarantees that this bound will hold for $k \geq \frac{9}{2}r^2$, but in Table~\ref{table:upperBounds}  we give a more precise lower bound on $k$, obtained computationally by searching for which values of $k < \frac{9}{2}r^2$ Theorem~\ref{t:upperBound} guarantees the desired bound on $\SP(2k+r,k)$.

\begin{table}[H]
\begin{center}
\begin{small}
\begin{tabular}{r||c|c|c|c|c|c|c|c}
  $r$ & 3 & 4 & 5 & 6 & 7 & 8 & 9 & 10 \\   \hline
  for $k \geq$ & 17 & 35 & 32 & 97 & 71 & 189 & 253 & 311 \\
  $\SP(2k+r,k) \leq$ & $2k+6$ & $2k+9$ & $2k+13$ & $2k+16$ & $2k+20$ & $2k+23$ & $2k+27$ & $2k+30$
\end{tabular}
\end{small}
\caption{Upper bounds on  $\SP(2k+r,k)$ and the values of $k$ for which they hold.}\label{table:upperBounds}
\end{center}
\end{table}

Finally in this section we prove Theorem~\ref{T:3ktake2case} by showing that, when $c=2$ and $r = k -2$, Lemma~\ref{l:mainConstruction}(b) and Theorem~\ref{t:upperBound} allow us to narrow $\SP(n,k)$ down to one of two possible values.

\begin{proof}[\textbf{\textup{Proof of Theorem~\ref{T:3ktake2case}}}]
Let $\ell \geq 2$ be the integer such that $k=2\ell$. First, suppose for a contradiction that $\SP(n,k) \geq \binom{n/2}{2}+2$. Using this, together with $n=6\ell-2$, $c=2$ and $r=2\ell-2$, Theorem~\ref{t:upperBound} implies that
\begin{equation}\label{E:upperBound2}
\ceil{\left(1- \tfrac{6\ell - 6}{6\ell - 2}\right)\left(\tbinom{3\ell - 1}{2}+2\right)} + \L_2\left(\floor{\tfrac{6\ell - 6}{6\ell - 2}\left(\tbinom{3\ell - 1}{2}+2\right)}\right) \leq 6\ell - 3.
\end{equation}
Observe that $\tceil{(1- \tfrac{6\ell - 6}{6\ell - 2})(\tbinom{3\ell - 1}{2}+2)}=\tceil{3\ell-2+\frac{4}{3\ell - 1}}=3\ell - 1$ and $\tfloor{\tfrac{6\ell - 6}{6\ell - 2}(\tbinom{3\ell - 1}{2}+2)} = \tfloor{\tbinom{3\ell-2}{2}+\tfrac{6\ell-6}{3\ell-1}}=\tbinom{3\ell-2}{2}+1$. So \eqref{E:upperBound2} is equivalent to $3\ell - 1 + \L_2\left(\tbinom{3\ell-2}{2}+1\right) \leq 6\ell-3$. Clearly $\L_2\left(\binom{3\ell-2}{2}+1\right)> 3\ell-2$, and hence we have a contradiction. Thus $\SP(n,k) \leq \binom{n/2}{2}+1$.

Now we proceed to show $\SP(n,k) \geq \binom{n/2}{2}$. Observe that $k-2 \equiv 0 \mod{2}$ since $k$ is even. Thus, by Lemma~\ref{l:mainConstruction}(b), with $m = 2$ and $h=\frac{n}{2}$, we know there exists an $(n,k)$-Sperner partition system with $p$ partitions, where \[p = \min\left\{\Bigl\lfloor\mfrac{n^2}{8}\Bigr\rfloor,\Bigl\lfloor\mfrac{2\binom{n/2}{3}}{k-2}\Bigr\rfloor\right\}.\]
Noting that $\frac{2}{k-2}\binom{n/2}{3} = \binom{n/2}{2}$ for $n = 3k-2$ and $\binom{n/2}{2}<\frac{1}{8}n^2$, it is apparent that $p = \binom{n/2}{2}$ and the result therefore follows.
\end{proof}

In \cite[Theorem~4.1]{LiMeagher}, it was shown that $\SP(3k-1,k) \geq 3k-1$ for all integers $k \geq 3$. Using \eqref{e:mono}, Theorem~\ref{T:3ktake2case} provides a substantial improvement to this bound for even integers $k \geq 6$.

\section{Proof of Theorem~\ref{t:k=3Asymptotics}(a)}\label{S:th4a}

In this section and the next we extend an approach detailed in \cite[Lemmas~11 and 10]{Paper1} to prove, respectively, Theorem~\ref{t:k=3Asymptotics}(a) and Theorem~\ref{t:k=3Asymptotics}(b). In this section we are interested in parameter sets $(n,k)$ such that $k$ is odd and $n \equiv k+1 \mod{2k}$, in accordance with the hypotheses of Theorem~\ref{t:k=3Asymptotics}(a). For a given parameter set $(n,k)$, our overall approach will be as follows. In Definition~\ref{D:IPk+1mod2k} we define an integer program $\mathcal{I}_{(n,k)}$ and then, in Lemma~\ref{L:r1IPBound}, show that we can use an optimal solution of $\mathcal{I}_{(n,k)}$ to construct an $(n,k)$-Sperner partition system whose size is the optimal value of $\mathcal{I}_{(n,k)}$. Next, in Lemma~\ref{L:QAymptoticsk+1mod2k}, we establish that an obvious upper bound on the optimal value of $\mathcal{I}_{(n,k)}$ is asymptotic to $\MMS(n,k)$ and we then finally prove Theorem~\ref{t:k=3Asymptotics}(a) by showing that $\mathcal{I}_{(n,k)}$ achieves an optimal value asymptotic to this upper bound. For the basic concepts and definitions of integer and linear programming we direct the reader to \cite{Padberg}.

We now introduce some definitions and notation that we will use extensively throughout this section. We let $d$ be the integer such that $c=2d+1$, that is, such that $n=(2d+1)k+1$. We will construct our Sperner partition systems on a set $X=X_1 \cup X_2$ where $X_1$ and $X_2$ are disjoint sets such that $|X_1|=|X_2|=\frac{n}{2}$. For each nonnegative integer $i$, let
\begin{align*}
  \E_{i} &= \{E \subseteq X: |E \cap X_1|=i, |E \cap X_2|=2d+1-i\} \\
  \E^{*}_{i} &= \{E \subseteq X: |E \cap X_1|=i, |E \cap X_2|=2d+2-i\}.
\end{align*}
Note that the elements of $\E_i$ are $c$-sets and the elements of $\E^{*}_{i}$ are $(c+1)$-sets. For each $\ell \in \{0,\ldots,d\}$ define $\e_\ell=\binom{n/2}{d-\ell}\binom{n/2}{d+1+\ell}$ so that we have $|\E_{d-\ell}|=|\E_{d+1+\ell}|=\e_\ell$ and for each $\ell \in \{0,\ldots,d+1\}$ define $\e^*_\ell=\binom{n/2}{d+1-\ell}\binom{n/2}{d+1+\ell}$ so that we have $|\E^*_{d+1-\ell}|=|\E^*_{d+1+\ell}|=\e^*_\ell$. Of the integers in $\{0,\ldots,d\}$, let $u$ be the smallest that satisfies $a(u) \leq (k-1)b(u)$ where, for $x \in \{0,\ldots,d\}$,
\[a(x) = 2\medop\sum_{\ell = x+1}^{d}\e_\ell \quad \text{and} \quad
b(x) = \e^*_0+\medop\sum_{\ell = 1}^{x}\e^*_\ell.\]
Let $Q$ be the largest even integer that is at most $\frac{1}{k-1}a(u)$ and also let
\begin{align*}
\A&=\medop\bigcup_{i \in I}\E_{i} & \mbox{where}&& I&=\{0, \ldots, d - u - 1\} \cup \{d+u+2, \ldots, 2d + 1\}\\
\B&=\medop\bigcup_{i \in I^*}\E^{*}_{i} &\mbox{where}&& I^*&=\{d+1 - u, \ldots, d+1+u\}.
\end{align*}
Let $\F = \A \cup \B$. It is not hard to see that $u \leq d-1$ since $a(d-1)=2\e_d<2\e_0<2\e^*_0\leq 2b(d-1)$ and $k \geq 3$. Note that $\A$ contains $c$-sets and $\B$ contains $(c+1)$-sets, and that the sets in $\B$ are more ``balanced'' between $X_1$ and $X_2$ than the sets in $\A$. Obviously, no set in $\B$ can be a subset of a set in $\A$. Furthermore, no set in $\A$ can be a subset of a set in $\B$ because $\max\{|A \cap X_1|,|A \cap X_2|\}\geq d+u+2 $ for each $A \in \A$ and $\max\{|B \cap X_1|,|B \cap X_2|\} \leq d+1+u$ for each $B \in \B$. Thus $(X,\F)$ is a clutter. Also observe that $|\A| = a(u)$ and $|\B| = b(u)$. Note that all of the notation we just defined is implicitly dependent on the values of $n$ and $k$. These values will be clear from context, so this should not cause confusion.

We will construct a Sperner partition system using the sets in $\F$. Note that each partition in such a system will contain $k-1$ sets from $\A$ and one set from $\B$ and hence such a system can have size at most $Q+1$. Our construction depends on finding up to $Q$ disjoint triples of sets from $\F$ such that for each triple $\{E_1,E_2,E_3\}$ we have $E_1 \in \B$, $E_2,E_3 \in \A$ and $\sum_{i=1}^3|E_i \cap X_w|=3d+2$ for each $w \in \{1,2\}$. We encode this task in the  integer program below. We define $\eta^*_0,\ldots,\eta^*_u$ to be the unique sequence of integers such that $\lfloor\frac{1}{2}\eta^*_0\rfloor+\sum_{\ell=1}^u\eta^*_\ell=\frac{1}{2}Q$ and, for some $x \in \{0,\ldots,u\}$, we have $\eta^*_\ell=\e^*_\ell$ for $\ell \in \{0,\ldots,x-1\}$, $0 \leq \eta^*_x < \e^*_x$, $\eta^*_\ell=0$ for $\ell \in \{x+1,\ldots,u\}$ and, if $x=0$, $\eta^*_0 \equiv \e^*_0 \mod{2}$. Such a sequence exists since $\frac{1}{2}Q = \lfloor\frac{1}{2(k-1)}a(u)\rfloor \leq \lfloor\frac{1}{2}b(u)\rfloor=\lfloor\frac{1}{2}\e^*_0\rfloor+\sum_{\ell=1}^u\e^*_\ell$.

\begin{definition}\label{D:IPk+1mod2k}
For integers $k \geq 3$ and $n>2k$ with $k$ odd and $n \equiv k+1 \mod{2k}$, define $\mathcal{I}_{(n,k)}$ to be the integer program on nonnegative integer variables $x_{i,j}$ for all $(i,j) \in \Phi$, where
    \[\Phi=\{(i,j): u+1 \leq i\leq j \leq d \mbox{ and } j-i \leq u\},\]
that maximises $2\sum_{(i,j) \in \Phi} x_{i,j}$ subject to
\begin{alignat}{3}
 \medop\sum_{(i,i+\ell) \in \Phi}x_{i,i+\ell} & \leq \eta^*_\ell \quad && \mbox{for all $\ell \in \{1,\ldots,u\}$}\label{E:C1d}\\[1mm]
 \medop\sum_{(i,i) \in \Phi}x_{i,i} & \leq \lfloor\tfrac{1}{2}\eta^*_0\rfloor \label{E:C2d}\\[1mm]
 \medop\sum_{(\ell,j) \in \Phi }x_{\ell,j}+\medop\sum_{(i,\ell) \in \Phi }x_{i,\ell} & \leq \e_\ell   \quad && \mbox{for all $\ell \in \{u+1,\ldots,d\}$}.\label{E:C3d}
\end{alignat}
\end{definition}

Note that taking each variable to be 0 in $\mathcal{I}_{(n,k)}$ satisfies all of the constraints and hence a feasible solution exists. Also, twice the sum of \eqref{E:C1d} for $\ell \in \{1,\ldots,u\}$ and \eqref{E:C2d} has the objective function of $\mathcal{I}_{(n,k)}$ as its left hand side and, by our definition of $\eta^*_0,\ldots,\eta^*_u$, $Q$ as its right hand side. Hence the optimal value of $\mathcal{I}_{(n,k)}$ is at most $Q$. Further, each variable must be bounded above, since it appears in the objective function with positive coefficient.

\begin{lemma}\label{L:r1IPBound}
Let $k \geq 3$ and $n>2k$ be integers with $k$ odd and $n \equiv k+1 \mod{2k}$, and let $p$ be the optimal value of $\mathcal{I}_{(n,k)}$. Then there exists a $(n,k)$-Sperner partition system with $p$ partitions.
\end{lemma}

\begin{proof}
Consider an arbitrary optimal solution $\{x_{i,j}:(i,j) \in \Phi\}$. This solution has objective value $p$ where $p \leq Q$. We will use this solution to create a partial edge colouring of the clutter $H=(X,\F)$ with $p$ colours and then apply Lemma~\ref{l:colouring} to construct a Sperner partition system.

Note that any permutation of $X_w$ is an automorphism of $H$ for each $w \in \{1,2\}$. Define a set of colours $C=\bigcup_{(i,j) \in \Phi}(C_{i,j} \cup C'_{i,j})$, where $|C_{i,j}|=|C'_{i,j}|=x_{i,j}$ for each $(i,j) \in \Phi$, and note that $|C|=p$. By Lemma~\ref{l:colouring} it suffices to find a partial edge colouring $\gamma_1$ of $H$ with colour set $C$ such that, for each $z \in C$, $|\gamma^{-1}_1(z)|=k$ and $\sum_{x \in X_w}\deg^{\gamma_1}_z(x)=kd+\frac{k+1}{2}$ for $w \in \{1,2\}$. We first create a partial edge colouring $\gamma_0$ of $H$ with three sets in each colour class which we will later extend to the desired colouring $\gamma_1$. We create this colouring $\gamma_0$ by beginning with all edges of $H$ uncoloured and then choosing certain edges to go in colour classes. We first describe this process and then justify that we can in fact perform it to obtain $\gamma_0$.

For each $(i,j) \in \Phi$, one at a time in arbitrary order, we proceed as follows. For each $z \in C_{i,j} \cup C'_{i,j}$ we assign colour $z$ to three previously uncoloured edges:
\begin{itemize}
    \item
one from each of $\E_{d-i}$, $\E_{d+1+j}$ and $\E^*_{d+1+i-j}$ if $z \in C_{i,j}$; and
    \item
one from each of $\E_{d-j}$, $\E_{d+1+i}$ and $\E^*_{d+1+j-i}$ if $z \in C'_{i,j}$.
\end{itemize}
Because $(i,j) \in \Phi$, it can be checked that all the edges we colour are in $\F=\A \cup \B$. Further, observe that we will have $\sum_{x \in X_w}\deg^{\gamma_0}_{z}(x)=3d+2$ for each $w \in \{1,2\}$ and $z \in C$.

After this process is completed for each $(i,j) \in \Phi$, we call the resulting colouring $\gamma_0$. We will be able to perform this process provided that we do not attempt to colour more than $|\E_{i}|$ sets in $\E_{i}$ for any $i \in I$ or more than $|\E^*_{i}|$ sets in $\E^*_{i}$ for any $i \in I^*$.

\begin{itemize}
    \item[(i)]
Let $\ell \in \{1,\ldots,u\}$. Each of the $\sum_{(i,i+\ell) \in \Phi}x_{i,i+\ell}$ colours in $\bigcup_{(i,i+\ell) \in \Phi}C_{i,i+\ell}$ is assigned to exactly one of the edges in $\E^*_{d+1-\ell}$ and no other colours are assigned to these edges. Similarly, each of the $\sum_{(i,i+\ell) \in \Phi}x_{i,i+\ell}$ colours in $\bigcup_{(i,i+\ell) \in \Phi}C'_{i,i+\ell}$ is assigned to exactly one of the edges in $\E^*_{d+1+\ell}$ and no other colours are assigned to these edges. Thus, since $\eta^*_\ell \leq \e^*_\ell = |\E^*_{d+1-\ell}|=|\E^*_{d+1+\ell}|$, we do not run out of sets in $\E^*_{d+1-\ell}$ or $\E^*_{d+1+\ell}$ by \eqref{E:C1d}.
\item[(ii)]
Each of the $2\sum_{(i,i) \in \Phi}x_{i,i}$ colours in $\bigcup_{(i,i) \in \Phi}(C_{i,i} \cup C'_{i,i})$ is assigned to exactly one of the edges in $\E^*_{d+1}$ and no other colours are assigned to these edges. Thus, since $\eta^*_0 \leq \e^*_0 =  |\E^*_{d+1}|$, we do not run out of sets in $\E_{d+1}$ by \eqref{E:C2d}.
    \item[(iii)]
Let $\ell \in \{u+1,\ldots,d\}$. Each of the $\sum_{(\ell,j) \in \Phi}x_{\ell,j}$ colours in $\bigcup_{(\ell,j) \in \Phi}C_{\ell,j}$ and each of the $\sum_{(i,\ell) \in \Phi}x_{i,\ell}$ colours in $\bigcup_{(i,\ell) \in \Phi}C'_{i,\ell}$ is assigned to exactly one of the edges in $\E_{d-\ell}$ and no other colours are assigned to these edges. Similarly each of the $\sum_{(i,\ell) \in \Phi}x_{i,\ell}$ colours in $\bigcup_{(i,\ell) \in \Phi}C_{i,\ell}$ and each of the $\sum_{(\ell,j) \in \Phi}x_{\ell,j}$ colours in $\bigcup_{(\ell,j) \in \Phi}C'_{\ell,j}$ is assigned to exactly one of the edges in $\E_{d+1+\ell}$ and no other colours are assigned to these edges. Thus, since $\e_\ell=|\E_{d-\ell}|=|\E_{d+1+\ell}|$, we do not run out of sets in $\E_{d-\ell}$ or $\E_{d+1+\ell}$ by \eqref{E:C3d}.
\end{itemize}

So the colouring $\gamma_0$ does indeed exist. We now extend $\gamma_0$ to the desired colouring $\gamma_1$. Note that if $k=3$ this process will be trivial and $\gamma_1$ will equal $\gamma_0$. Let $\A^\dag$ be the set of all edges in $\A$ that are not coloured by $\gamma_0$. Because $\gamma_0$ has $p$ colour classes, each containing two edges in $\A$, we have $|\A^\dag|=|\A|-2p$. Now $|\A| \geq p(k-1)$ since $|\A|=a(u)$ and $p \leq Q \leq \frac{1}{k-1}a(u)$. Thus $|\A^\dag| \geq p(k-3)$. For each $\ell \in \{u+1,\ldots,d\}$ we have $|\A^\dag \cap \E_{d-\ell}|=|\A^\dag \cap \E_{d+1+\ell}|$ by the way we created $\gamma_0$. Thus we can create a partition $\A^\ddag$ of $\A^\dag$ into pairs such that for each pair $\{E,E'\}$ we have $E \in \E_{d-\ell}$ and $E' \in \E_{d+1+\ell}$ for some $\ell \in \{u+1,\ldots,d\}$. We form $\gamma_1$ from $\gamma_0$ by adding to each colour class the edges from $\frac{k-3}{2}$ pairs in $\A^\ddag$ in such a way that no pair is allocated to two different colour classes. This is possible because $|\A^\ddag|=\frac{1}{2}|\A^\dag|\geq p\frac{k-3}{2}$.
We claim that $\gamma_1$ has the required properties. To see this, note that $|\gamma^{-1}_1(z)|=k$ for each $z \in C$ because each colour class in $\gamma_0$ contained 3 edges and had $k-3$ edges added to it to form $\gamma_1$. Further, for each $z \in C$ and $w \in \{1,2\}$, $\sum_{x \in X_w}\deg^{\gamma_1}_z(x)=kd+\frac{k+1}{2}$ because $\sum_{x \in X_w}\deg^{\gamma_0}_z(x)=3d+2$ and $|E \cap X_w|+|E' \cap X_w|=2d+1$ for each pair $\{E,E'\}$ in $\B^\ddag$.
\end{proof}

Next we show that $Q$ is asymptotic to $\MMS(n,k)$. We shall require an easy consequence of Stirling's approximation (see \cite{RomikStirling} for example), namely that as $x$ and $y$ tend to infinity with $y \leq \frac{x}{2}$,
\begin{equation}\label{E:StirlingApprox}
\mbinom{x}{y} \sim A(x,y) \quad \text{where} \quad A(x,y) = \mfrac{x^{x+1/2}}{(2\pi)^{1/2}y^{y+1/2}(x-y)^{x-y+1/2}}.
\end{equation}

\begin{lemma}\label{L:QAymptoticsk+1mod2k}
Let $n$ and $k$ be integers such that $n \rightarrow \infty$ with $n \equiv k+1 \mod{2k}$, $k=o(n)$, and $k \geq 3$ is odd. Then $Q \sim \MMS(n,k)$.
\end{lemma}

\begin{proof}
Observe that $d \rightarrow \infty$ since $k=o(n)$. Furthermore, as $c=2d+1$ and $r=1$ in this case, we have
\begin{equation}\label{e:mmsexpr}
\MMS(n,k)=\mfrac{1}{k-1+\frac{2d+2}{n-2d-1}}\mbinom{n}{2d+1}.
\end{equation}
Recall that $(X,\F)$ is a clutter with $a(u)$ edges of size $2d+1$ and $b(u)$ edges of size $2d+2$. So, by the LYM-inequality (see \cite[p. 25]{FraTok}),  $a(u)/\binom{n}{2d+1}+b(u)/\binom{n}{2d+2} \leq 1$ or, equivalently, $a(u)+\frac{2d+2}{n-2d-1}b(u) \leq \binom{n}{2d+1}$. Thus, because $\frac{1}{k-1}a(u) \leq b(u)$ by the definition of $u$,
\[Q \leq \mfrac{a(u)}{k-1}  \leq \mfrac{1}{k-1+\frac{2d+2}{n-2d-1}}\mbinom{n}{2d+1} = \MMS(n,k).\]
So it remains to show that $Q \geq \MMS(n,k)(1-o(1))$.


For technical reasons, we extend the definitions of $a$, $b$ and $\e$ given at the start of this section by defining $a(-1)=\binom{n}{2d+1}$, $b(-1)=0$ and $\e_{-1}=0$. Using the definitions of $Q$ and $a$,
\begin{equation}\label{e:adiff}
Q > \mfrac{a(u)}{k-1}-2 = \mfrac{a(u-1)-2\e_{u}}{k-1}.
\end{equation}
We will bound $Q$ by bounding $a(u-1)$ below and applying \eqref{e:adiff}.

We first show that
\begin{equation}\label{e:coverage2}
(2d+2)b(u-1) = (n-2d-1)\left(\mbinom{n}{2d+1}-a(u-1)\right)-(n-2d-2u)\e_{u-1}.
\end{equation}
We may assume $u \geq 1$ for otherwise $u=0$, $a(u-1)=\binom{n}{2d+1}$, $b(u-1)=0$, $\e_{u-1}=0$ and \eqref{e:coverage2} holds.
Define
\begin{align*}
\mathcal{C}&=\medop\bigcup_{i \in I}\E_{i}, &\mbox{where}&& I&=\{0, \ldots, d - u\} \cup \{d+u+1, \ldots, 2d + 1\} \\
\mathcal{D}&=\medop\bigcup_{i \in I^*}\E^{*}_{i}, &\mbox{where}&& I^*&=\{d - u+2, \ldots d+u\}\\
\overline{\mathcal{C}}&=\medop\bigcup_{i \in \overline{I}}\E_{i}, &\mbox{where}&& \overline{I}&=\{d-u+1, \ldots d+u\}.
\end{align*}
Thus we have $\overline{\mathcal{C}}=\binom{X}{2d+1} \setminus \mathcal{C}$, $|\mathcal{C}|=a(u-1)$, $|\mathcal{D}|=b(u-1)$ and $|\overline{\mathcal{C}}|=\binom{n}{2d+1}-a(u-1)$. We now count, in two ways, the number of pairs $(S,B)$ such that $S \in \overline{\mathcal{C}}$, $B \in \mathcal{D}$ and $S \subseteq B$.
\begin{itemize}
    \item
Each of the $b(u-1)$ sets in $\mathcal{D}$ has exactly $2d+2$ subsets in $\binom{X}{2d+1}$ and each of these is in $\overline{\mathcal{C}}$, because no set in $\mathcal{C}$ is a subset of a set in $\mathcal{D}$.
    \item
Each of the $\binom{n}{2d+1}-a(u-1)$ sets in $\overline{\mathcal{C}}$ has $n-2d-1$ supersets in $\binom{X}{2d+2}$. For each $S \in \overline{\mathcal{C}} \setminus {(\E_{d-u+1} \cup \E_{d+u})}$, all of these supersets of $S$ are in $\mathcal{D}$. For each of the $2\e_{u-1}$ sets $S \in \E_{d-u+1} \cup \E_{d+u}$, exactly $\frac{n}{2}-d-u$ of these supersets of $S$ are not in $\mathcal{D}$.
\end{itemize}
Equating our two counts, we see that \eqref{e:coverage2} does indeed hold.

By definition of $u$, $\frac{1}{k-1}a(u-1) > b(u-1)$. Substituting this into \eqref{e:coverage2} and solving for $\frac{1}{k-1}a(u-1)$ we see
\begin{equation}\label{e:coverage3}
\mfrac{a(u-1)}{k-1} > \frac{\binom{n}{2d+1}-\frac{n-2d-2u}{n-2d-1}\e_{u-1}}{k-1+\frac{2d+2}{n-2d-1}}.
\end{equation}
Substituting \eqref{e:coverage3} into \eqref{e:adiff} and then manipulating,  we obtain
\begin{equation}\label{e:finalmess}
Q > \frac{\binom{n}{2d+1}-\frac{n-2d-2u}{n-2d-1}\e_{u-1}}{k-1+\frac{2d+2}{n-2d-1}}-\mfrac{2\e_{u}}{k-1}-2 = \frac{\binom{n}{2d+1}-\frac{n-2d-2u}{n-2d-1}\e_{u-1}-2(1+\frac{2d+2}{(k-1)(n-2d-1)})\e_{u}}{k-1+\frac{2d+2}{n-2d-1}}-2
\end{equation}
Observing that the coefficients of $\e_{u-1}$ and $\e_{u}$ in the numerator of the final expression above are clearly $O(1)$ and that $\e_{u} \leq \e_{u-1} \leq \e_0$, we see from \eqref{e:mmsexpr} that \eqref{e:finalmess} implies that $Q > \MMS(n,k)(1-o(1))$ provided that $\e_0=o(\binom{n}{2d+1})$. To see this is the case, observe that
\[
\frac{\e_0}{\binom{n}{2d+1}} \sim \mfrac{A(\tfrac{n}{2},d)A(\tfrac{n}{2},d+1)}{A(n,2d+1)} \leq \mfrac{\left(A(\tfrac{n}{2},d+\tfrac{1}{2})\right)^2}{A(n,2d+1)} = \sqrt{\tfrac{2n}{\pi(2d+1)(n-2d-1)}} \sim \sqrt{\tfrac{k}{\pi d(k-1)}} = o(1).
\]
where the first $\sim$ uses \eqref{E:StirlingApprox} and the last was obtained using $n \sim 2dk$. So the proof is complete.
\end{proof}

We define the \emph{slack} in an inequality $f \leq g$ to be $g-f$. We will refer to this particularly in the case of the constraints \eqref{E:C1d}--\eqref{E:C3d} and, in the next section, the constraints \eqref{E:C1}--\eqref{E:C3}. Note a candidate solution to an integer program is feasible if and only if the slack in each of its constraints is nonnegative. Our next result shows that the optimal value of $\mathcal{I}_{(n,k)}$ is close to $Q$.

\begin{lemma}\label{L:IPachieves}
Let $k \geq 3$ and $n>2k$ be integers with $k$ odd and $n \equiv k+1 \mod{2k}$. The optimal value of $\mathcal{I}_{(n,k)}$ is at least $Q-2\binom{u}{2}-\frac{2(d-u+1)}{k-1}$.
\end{lemma}

\begin{proof}
For $t \in \{1,\ldots,u\}$, let $\beta_t$ be the slack in the constraint of $\mathcal{I}_{(n,k)}$ given by setting $\ell=t$ in \eqref{E:C1d}. Also let $\beta_0$ be the slack in the constraint \eqref{E:C2d} of $\mathcal{I}_{(n,k)}$. Similarly, for $t \in \{u+1,\ldots,d\}$, let $\alpha_{t}$ be the slack in the constraint of $\mathcal{I}_{(n,k)}$ given by setting $\ell=t$ in \eqref{E:C3d}. We will create a solution to $\mathcal{I}_{(n,k)}$ with objective value at least $Q-2\binom{u}{2}-\frac{2(d-u+1)}{k-1}$. We do this by beginning with the solution to $\mathcal{I}_{(n,k)}$ in which all the variables are 0 and iteratively improving the solution. Each step of the iteration proceeds as follows.
\begin{itemize}
    \item[(i)]
Take the existing solution. Let $y$ be the largest element of $\{1,\ldots,u\}$ for which $\beta_{y}\geq y$ if such an element exists. If no such element exists, let $y=0$ if $\beta_0\geq\frac{d-u+1}{k-1}$ and otherwise terminate the procedure.
    \item[(ii)]
Let $z$ be the largest element of $\{u+1,\ldots,d\}$ for which $\alpha_{z}\geq \delta$, where $\delta=1$ if $y \geq 1$ and $\delta=2$ if $y=0$. We claim that $z$ exists, $z \geq u+2y$ and, if $y \geq 1$, then $\alpha_{\ell}\geq 1$ for each $\ell \in \{z-2y+1,\ldots,z\}$.
    \item[(iii)]
If $y \geq 1$, increase the value of each of the $y$ variables in $\{x_{z-i,z-y-i}: i \in \{0,\ldots,y-1\}\}$ by 1. This results in $\beta_{y}$ decreasing by $y$, $\alpha_{\ell}$ decreasing by 1 for each $\ell \in \{z-2y+1,\ldots,z\}$, and all other $\alpha_\ell$ and $\beta_\ell$ remaining unchanged.
    \item[(iv)]
If $y = 0$, increase the value of the variable $x_{z,z}$ by 1. This results in $\beta_0$ decreasing by $1$, $\alpha_{z}$ decreasing by 2, and all other $\alpha_\ell$ and $\beta_\ell$ remaining unchanged.
\end{itemize}
Provided the claim in (ii) holds, it can be seen that this procedure will terminate with a solution in which $\beta_0\leq\frac{d-u+1}{k-1}$ and $\beta_{\ell}\leq\ell-1$ for each $\ell \in \{1,\ldots,u\}$. As noted just below Definition~\ref{D:IPk+1mod2k}, twice the sum of the constraints \eqref{E:C1d} for $\ell \in \{1,\ldots,u\}$ and \eqref{E:C2d} has the objective function of $\mathcal{I}_{(n,k)}$ as its left hand side and $Q$ as its right hand side. Thus this solution will have objective value at least $Q-2\binom{u}{2}-\frac{2(d-u+1)}{k-1}$ since
\[\medop\sum_{\ell=0}^u \beta_\ell \leq \mfrac{d-u+1}{k-1} +\medop\sum_{\ell=l}^u (\ell-1) = \mbinom{u}{2}+\mfrac{d-u+1}{k-1}.\]
So it suffices to show that the claim in (ii) holds in each step.

Throughout the process none of the $\alpha_\ell$ and $\beta_\ell$ ever increase and thus the values of $y$ and $z$ chosen at each step form nonincreasing sequences.
Let $\alpha=\sum_{\ell=u+1}^{d} \alpha_\ell$ and $\beta= \sum_{\ell=0}^u \beta_\ell$.
At the beginning of the process, $\alpha=\frac{1}{2}a(u)$ and $\beta=\frac{1}{2}Q$ and we have $a(u) \geq (k-1)Q$ by the definition of $Q$. Further, at each step of the process the reduction in $\alpha$ is exactly twice the reduction in $\beta$. So, because $k-1\geq 2$, at any step of the process $\alpha\geq(k-1)\beta$. Fix a step and the values of $y$ and $z$ at this step. We will show the claim in (ii) holds for this step
by considering cases according to whether $y=0$.

\textbf{Case 1.} Suppose that $y\geq1$. Since $\eta^*_y \geq \beta_y > 0$, we have $\eta^*_\ell= \e^*_\ell$ for each $\ell \in \{1,\ldots,y-1\}$ by our definition of $\eta^*_0,\ldots,\eta^*_u$.  Further, for each $\ell \in \{1,\ldots,y-1\}$, $\beta_{\ell}$ has not so far been decreased and hence $\beta_\ell=\eta^*_\ell=\e^*_\ell$. Thus
\begin{equation}\label{E:alphaLowerBound}
\alpha \geq (k-1)\beta \geq \tfrac{k-1}{2}\e^*_0+(k-1)\medop\sum_{\ell=1}^{y-1}\e^*_\ell.
\end{equation}
In particular $\alpha$ is positive and hence $z$ exists. By our choice of $z$, $\alpha_{\ell}= 0$ for each $\ell \in \{z+1,\ldots,d\}$. Thus, if $z < u+2y$ then we would have
\[\alpha \leq \medop\sum_{\ell=u+1}^{u+2y-1}\e_{\ell} < \medop\sum_{\ell=u+1}^{u+2y-1}\e^*_\ell\]
where the second inequality follows from the definitions of $\e_{\ell}$ and $\e^*_\ell$ using the fact that $d+1<\frac{n}{2}$. But this can be seen to contradict \eqref{E:alphaLowerBound} by first using $k-1 \geq 2$ in \eqref{E:alphaLowerBound} and then applying $2y-1$ times the fact that $\e^*_{\ell_1} > \e^*_{\ell_2}$ for any integers $\ell_1$ and $\ell_2$ with $0 \leq \ell_1 < \ell_2 \leq d+1$. So $z \geq u+2y$. Finally note that, for any $\ell \in \{z-2y+1,\ldots,z\}$, any previous step of the process which decreased the slack in $\alpha_{\ell}$ also decreased the slack in $\alpha_{z}$ by an equal amount (namely 1). Thus, because at the start of the process $\alpha_{\ell}>\alpha_{z}$, this still holds at the present step and hence $\alpha_{\ell}\geq 1$. So the claim is proved.

\textbf{Case 2.} Suppose that $y=0$. Then $\beta_0 \geq \frac{d-u+1}{k-1}$ by our choice of $y$, so $\beta \geq \frac{d-u+1}{k-1}$ and hence $\alpha \geq d-u+1$. So, by the pigeonhole principle, $\alpha_{\ell}$ has slack at least $2$ for some $\ell \in \{u+1,\ldots,d\}$. Thus $z$ exists, and $z \geq u+2y=u$ trivially. So again the claim is proved.
\end{proof}

In the above result, note that $2\binom{u}{2}+\frac{2(d-u+1)}{k-1}$ is obviously $O(d^2)$ since $u \leq d$. 

\begin{proof}[\textbf{\textup{Proof of Theorem~\ref{t:k=3Asymptotics}(a).}}]
This follows from Lemmas~\ref{L:r1IPBound}, \ref{L:QAymptoticsk+1mod2k} and \ref{L:IPachieves}, noting that in the last of these our lower bound on the optimal value of the integer program is $Q-O(d^2)$ and hence is asymptotic to $Q$ because $Q \sim \MMS(n,k)$ and clearly $d^2=o(\MMS(n,k))$.
\end{proof}

\section{Proof of Theorem~\ref{t:k=3Asymptotics}(b)}\label{S:th4b}

In this section we are interested in parameter sets $(n,k)$ such that $k$ is odd and $n \equiv k-1 \mod{2k}$, in accordance with the hypotheses of Theorem~\ref{t:k=3Asymptotics}(b). For a given parameter set $(n,k)$, our overall approach in this section is similar to that of the last section. In Definition~\ref{D:IPk-1mod2k} we define an integer program $\mathcal{I}_{(n,k)}$ and then, in Lemma~\ref{L:rkIPBound}, show that we can use an optimal solution of $\mathcal{I}_{(n,k)}$ to construct an $(n,k)$-Sperner partition system whose size is the optimal value of $\mathcal{I}_{(n,k)}$. Next, in Lemma~\ref{L:QAymptoticsk-1mod2k}, we establish that an obvious upper bound on the optimal value of $\mathcal{I}_{(n,k)}$ is asymptotic to $\MMS(n,k)$ and we then finally prove Theorem~\ref{t:k=3Asymptotics}(b) by showing that, for $k \geq 5$, $\mathcal{I}_{(n,k)}$ achieves an optimal value asymptotic to this upper bound. Our proof of this last result differs substantially from the proof of the analogous result in Section~\ref{S:th4a}, however.  

We do not retain any of the notation defined in the last section. Instead, we will redefine most of it to suit our purposes here. Throughout we will let $d$ be the integer such that $c=2d$, that is, such that $n=(2d+1)k-1$. Again, we will construct our Sperner partition systems on a set $X=X_1 \cup X_2$ where $X_1$ and $X_2$ are disjoint sets such that $|X_1|=|X_2|=\frac{n}{2}$. For each nonnegative integer $i$, let
\begin{align*}
  \E_{i} &= \{E \subseteq X: |E \cap X_1|=i, |E \cap X_2|=2d-i\} \\
  \E^{*}_{i} &= \{E \subseteq X: |E \cap X_1|=i, |E \cap X_2|=2d+1-i\}.
\end{align*}
Note that the elements of $\E_i$ are $c$-sets and the elements of $\E^{*}_{i}$ are $(c+1)$-sets. In fact, when considered in terms of $c$, these definitions are the same as those given in the last section although they appear different when phrased in terms of $d$. For each $\ell \in \{0,\ldots,d\}$, define $\e_\ell=\binom{n/2}{d-\ell}\binom{n/2}{d+\ell}$ and $\e^*_\ell=\binom{n/2}{d-\ell}\binom{n/2}{d+1+\ell}$ so that we have $|\E_{d-\ell}|=|\E_{d+\ell}|=\e_\ell$ and $|\E^*_{d-\ell}|=|\E^*_{d+1+\ell}|=\e^*_\ell$. Of the integers in $\{-1,\ldots,d-1\}$, let $u$ be the largest that satisfies $(k-1)a(u) \leq b(u)$ where $a(-1)=0$, $b(-1)=\binom{n}{2d+1}$ and, for $x \in \{0,\ldots,d\}$,
\[a(x) = \e_0+2\medop\sum_{\ell = 1}^{x}\e_\ell \quad \text{and} \quad
b(x) = 2\medop\sum_{\ell = x+1}^{d}\e^*_\ell.\]
Let $Q=a(u)$ if $a(u)$ is even and $Q=a(u)-1$ if $a(u)$ is odd. Also let
\begin{align*}
\A&=\medop\bigcup_{i \in I}\E_{i}, & \mbox{where } I&=\{d - u, \ldots d+u\} \\
\B&=\medop\bigcup_{i \in I^*}\E^{*}_{i}, &\mbox{where } I^*&=\{0, \ldots, d - u - 1\} \cup \{d+u+2, \ldots, 2d + 1\}.
\end{align*}
We allow $u$ to equal $-1$ in our definition above to ensure it is always well-defined. When $u=-1$, we have that $\A$ is empty and Definition~\ref{D:IPk-1mod2k} and Lemma~\ref{L:rkIPBound} below are trivial.  However our main interest in this section is in the regime where $n$ is large compared to $k$ and for these cases Lemma~\ref{L:QAymptoticsk-1mod2k} below implies that $u \geq 0$ and the systems we construct are nontrivial. Let $\F = \A \cup \B$. As in the previous section, $\A$ contains $c$-sets and $\B$ contains $(c+1)$-sets but, unlike in the previous section, here the sets in $\A$ are more ``balanced'' between $X_1$ and $X_2$ than those in $\B$. Obviously, no set in $\B$ can be a subset of a set in $\A$. Furthermore, no set in $\A$ can be a subset of a set in $\B$ because $\min\{|A \cap X_1|,|A \cap X_2|\} \geq d-u$ for each $A \in \A$ and $\min\{|B \cap X_1|,|B \cap X_2|\} \leq d-u-1$ for each $B \in \B$. Thus $(X,\F)$ is a clutter. Also observe that $|\A| = a(u)$ and $|\B| = b(u)$. Again, all of the notation just defined is implicitly dependent on the values of $(n,k)$.

We will construct a Sperner partition system using the sets in $\F$. Note that each partition in such a system will contain one set from $\A$ and $k-1$ sets from $\B$ and hence such a system can have size at most $Q+1$. Similarly to the previous section, our construction here depends on finding up to $Q$ disjoint triples of sets from $\F$. Here, for each triple $\{E_1,E_2,E_3\}$, we will have $E_1 \in \A$, $E_2,E_3 \in \B$ and $\sum_{i=1}^3|E_i \cap X_w|=3d+1$ for each $w \in \{1,2\}$. Again, we encode this task in an integer program.

\begin{definition}\label{D:IPk-1mod2k}
For integers $k \geq 3$ and $n>2k$ with $k$ odd and $n \equiv k-1 \mod{2k}$, define $\mathcal{I}_{(n,k)}$ to be the integer program on nonnegative integer variables $x_{i,j}$ for all $(i,j) \in \Phi$, where
    \[\Phi=\{(i,j): u+1 \leq i\leq j \leq d \mbox{ and } j-i \leq u\},\]
that maximises $2\sum_{(i,j) \in \Phi} x_{i,j}$ subject to
\begin{alignat}{3}
 \medop\sum_{(i,i+\ell) \in \Phi}x_{i,i+\ell} & \leq \e_\ell \quad && \mbox{for all $\ell \in \{1,\ldots,u\}$}\label{E:C1}\\[1mm]
 \medop\sum_{(i,i) \in \Phi}x_{i,i} & \leq \lfloor\tfrac{1}{2}\e_0\rfloor\label{E:C2}\\[1mm]
 \medop\sum_{(\ell,j) \in \Phi }x_{\ell,j}+\medop\sum_{(i,\ell) \in \Phi }x_{i,\ell} & \leq \e^*_\ell   \quad && \mbox{for all $\ell \in \{u+1,\ldots,d\}$}.\label{E:C3}
\end{alignat}
\end{definition}

When $u=-1$ we have $\Phi=\emptyset$ in the above and we consider $\mathcal{I}_{(n,k)}$ to be a trivial program with optimal value $0$. Again, taking each variable to be 0 in $\mathcal{I}_{(n,k)}$ satisfies all of the constraints and hence a feasible solution exists. Also, twice the sum of \eqref{E:C1} for $\ell \in \{1,\ldots,u\}$ and \eqref{E:C2} has the objective function of $\mathcal{I}_{(n,k)}$ as its left hand side and $Q=2\lfloor\frac{1}{2}a(u)\rfloor$ as its right hand side. Hence the optimal value of $\mathcal{I}_{(n,k)}$ is at most $Q$ and once again each variable must be bounded above.

\begin{lemma}\label{L:rkIPBound}
Let $k \geq 3$ and $n>2k$ be integers with $k$ odd and $n \equiv k-1 \mod{2k}$, and let $p$ be the optimal value of $\mathcal{I}_{(n,k)}$. Then there exists a $(n,k)$-Sperner partition system with $p$ partitions.
\end{lemma}

\begin{proof}
The result is trivial if $u=-1$, so we may assume that $u \geq 0$ and hence $\A$ is nonempty. Consider an arbitrary optimal solution $\{x_{i,j}:(i,j) \in \Phi\}$. This solution has objective value $p$ where $p \leq Q$. We will use this solution to create a partial edge colouring of the clutter $H=(X,\F)$ with $p$ colours and then apply Lemma~\ref{l:colouring} to construct a Sperner partition system.

Note that any permutation of $X_w$ is an automorphism of $H$ for each $w \in \{1,2\}$. Define a set of colours $C=\bigcup_{(i,j) \in \Phi}(C_{i,j} \cup C'_{i,j})$, where $|C_{i,j}|=|C'_{i,j}|=x_{i,j}$ for each $(i,j) \in \Phi$, and note that $|C|=p$. By Lemma~\ref{l:colouring} it suffices to find a partial edge colouring $\gamma_1$ of $H$ with colour set $C$ such that, for each $z \in C$, $|\gamma^{-1}_1(z)|=k$ and $\sum_{x \in X_w}\deg^{\gamma_1}_z(x)=kd+\frac{k-1}{2}$ for $w \in \{1,2\}$. We first create a partial edge colouring $\gamma_0$ of $H$ with three sets in each colour class which we will later extend to the desired colouring $\gamma_1$. We create this colouring $\gamma_0$ by beginning with all edges of $H$ uncoloured and then choosing certain edges to go in colour classes. We first describe this process and then justify that we can in fact perform it to obtain $\gamma_0$.

For each $(i,j) \in \Phi$ one at a time in arbitrary order we proceed as follows. For each $z \in C_{i,j} \cup C'_{i,j}$ we assign colour $z$ to three previously uncoloured edges:
\begin{itemize}
    \item
one from each of $\E^*_{d-i}$, $\E^*_{d+1+j}$ and $\E_{d+i-j}$ if $z \in C_{i,j}$; and
    \item
one from each of $\E^*_{d-j}$, $\E^*_{d+1+i}$ and $\E_{d+j-i}$ if $z \in C'_{i,j}$.
\end{itemize}
Because $(i,j) \in \Phi$, it can be checked that all the edges we colour are in $\F=\A \cup \B$. Further, observe that we will have $\sum_{w \in X_w}\deg^{\gamma_0}_{z}(x)=3d+1$ for each $w \in \{1,2\}$ and $z \in C$.

After this process is completed for each $(i,j) \in \Phi$, call the resulting colouring $\gamma_0$. We will be able to perform this process provided that we do not attempt to colour more than $|\E_{i}|$ sets in $\E_{i}$ for any $i \in I$ or more than $|\E^*_{i}|$ sets in $\E^*_{i}$ for any $i \in I^*$.

\begin{itemize}
    \item[(i)]
Let $\ell \in \{1,\ldots,u\}$. Each of the $\sum_{(i,i+\ell) \in \Phi}x_{i,i+\ell}$ colours in $\bigcup_{(i,i+\ell) \in \Phi}C_{i,i+\ell}$ is assigned to exactly one of the edges in $\E_{d-\ell}$ and no other colours are assigned to these edges. Similarly, each of the $\sum_{(i,i+\ell) \in \Phi}x_{i,i+\ell}$ colours in $\bigcup_{(i,i+\ell) \in \Phi}C'_{i,i+\ell}$ is assigned to exactly one of the edges in $\E_{d+\ell}$ and no other colours are assigned to these edges. Thus, since $|\E_{d-\ell}|=|\E_{d+\ell}|=\e_\ell$, we do not run out of sets in $\E_{d-\ell}$ or $\E_{d+\ell}$ by \eqref{E:C1}.
\item[(ii)]
Each of the $2\sum_{(i,i) \in \Phi}x_{i,i}$ colours in $\bigcup_{(i,i) \in \Phi}(C_{i,i} \cup C'_{i,i})$ is assigned to exactly one of the edges in $\E_{d}$ and no other colours are assigned to these edges. Thus, since $|\E_{d}|=\e_0$, we do not run out of sets in $\E_{d}$ by \eqref{E:C2}.
    \item[(iii)]
Let $\ell \in \{u+1,\ldots,d\}$. Each of the $\sum_{(\ell,j) \in \Phi}x_{\ell,j}$ colours in $\bigcup_{(\ell,j) \in \Phi}C_{\ell,j}$ and each of the $\sum_{(i,\ell) \in \Phi}x_{i,\ell}$ colours in $\bigcup_{(i,\ell) \in \Phi}C'_{i,\ell}$ is assigned to exactly one of the edges in $\E^*_{d-\ell}$, and no other colours are assigned to these edges. Similarly, each of the $\sum_{(i,\ell) \in \Phi}x_{i,\ell}$ colours in $\bigcup_{(i,\ell) \in \Phi}C_{i,\ell}$ and each of the $\sum_{(\ell,j) \in \Phi}x_{\ell,j}$ colours in $\bigcup_{(\ell,j) \in \Phi}C'_{\ell,j}$ is assigned to exactly one of the edges in $\E^*_{d+1+\ell}$, and no other colours are assigned to these edges. Thus, since $|\E^*_{d-\ell}|=|\E^*_{d+1+\ell}|=\e^*_\ell$, we do not run out of sets in $\E^*_{d-\ell}$ or $\E^*_{d+1+\ell}$ by \eqref{E:C3}.
\end{itemize}

So the colouring $\gamma_0$ does indeed exist. We now extend $\gamma_0$ to the desired colouring $\gamma_1$. Note that if $k=3$ this process will be trivial and $\gamma_1$ will equal $\gamma_0$. Let $\B^\dag$ be the set of all edges in $\B$ that are not coloured by $\gamma_0$. Because $\gamma_0$ has $p$ colour classes, each containing two edges in $\B$, we have $|\B^\dag|=|\B|-2p$. Now $|\B| \geq (k-1)|\A| \geq p(k-1)$ since $|\A|=a(u)$, $|\B|=b(u)$ and $p \leq Q \leq a(u)$. Thus $|\B^\dag| \geq p(k-3)$. For each $\ell \in \{u+1,\ldots,d\}$ we have $|\B^\dag \cap \E^*_{d-\ell}|=|\B^\dag \cap \E^*_{d+1+\ell}|$ by way we defined $\gamma_0$. Thus we can create a partition $\B^\ddag$ of $\B^\dag$ into pairs such that for each pair $\{E,E'\}$ we have $E \in \E^*_{d-\ell}$ and $E' \in \E^*_{d+1+\ell}$ for some $\ell \in \{u+1,\ldots,d\}$. We form $\gamma_1$ from $\gamma_0$ by adding to each colour class the edges from $\frac{k-3}{2}$ pairs in $\B^\ddag$ in such a way that no pair is allocated to two different colour classes. This is possible because $|\B^\ddag|=\frac{1}{2}|\B^\dag| \geq p\frac{k-3}{2}$. We claim that $\gamma_1$ has the required properties. To see this, note that $|\gamma^{-1}_1(z)|=k$ for each $z \in C$ because each colour class in $\gamma_0$ contained 3 edges and had $k-3$ edges added to it to form $\gamma_1$. Further, for each $z \in C$ and $w \in \{1,2\}$, $\sum_{x \in X_w}\deg^{\gamma_1}_z(x)=kd+\frac{k-1}{2}$ because $\sum_{x \in X_w}\deg^{\gamma_0}_z(x)=3d+1$ and $|E \cap X_w|+|E' \cap X_w|=2d+1$ for each pair $\{E,E'\}$ in $\B^\ddag$.
\end{proof}

Next we show that $Q$ is asymptotic to $\MMS(n,k)$. We do this in a different fashion to the proof of Lemma~\ref{L:QAymptoticsk+1mod2k}, one that allows us to also obtain an estimate of $u$ and some other technical results that will be required in our proof of Theorem~\ref{t:k=3Asymptotics}(b). Recall that the error function, denoted $\erf$, is defined for any real number $x$ by $\erf(x)=\int_{0}^{x}\exp(-t^2)\,dt$.

\begin{lemma}\label{L:QAymptoticsk-1mod2k}
Let $n$ and $k$ be integers such that $n \rightarrow \infty$ with $n \equiv k-1 \mod{2k}$, $k=o(n)$, and $k \geq 3$ is odd. Then
\begin{itemize}
    \item[\textup{(a)}]
$\e^*_\ell \sim (k-1)\e_\ell$ for each integer $\ell \geq 0$ such that $\ell=O(\sqrt{d})$;
    \item[\textup{(b)}]
$\e_\ell \sim \e_0\exp(-\frac{k}{d(k-1)}\ell^2)$ for each integer $\ell \geq 0$ such that $\ell=O(\sqrt{d})$;
    \item[\textup{(c)}]
$u \sim \erf^{-1}(\frac{1}{2})\sqrt{d(k-1)/k}$; and
    \item[\textup{(d)}]
$Q \sim \MMS(n,k)$;
\end{itemize}
\end{lemma}

\begin{proof}
Let $n \equiv k-1 \mod{2k}$ be a positive integer. Observe that $d \rightarrow \infty$ since $k=o(n)$. It will often be useful to note that $n-2d=(2d+1)(k-1)$.

For each integer $\ell \geq 0$ such that $\ell=O(\sqrt{d})$,
\[\e^*_\ell = \mfrac{n-2d-2\ell}{2d+2\ell+2}\e_\ell \sim (k-1)\e_\ell\]
where the equality follows by the definitions of $\e_\ell$ and $\e^*_\ell$ and the $\sim$ follows because $n-2d=(2d+1)(k-1)$ and $\ell=O(\sqrt{d})$. So (a) holds.

For each integer $\ell \geq 0$ such that $\ell=O(\sqrt{d})$,
\begin{align*}
\frac{\e_\ell}{\e_0}\sim{}& \frac{d^{2d+1}(\frac{n}{2}-d)^{n-2d+1}}{(d-\ell)^{d-\ell+1/2}(d+\ell)^{d+\ell+1/2}(\frac{n}{2}-d-\ell)^{(n+1)/2-d-\ell}(\frac{n}{2}-d+\ell)^{(n+1)/2-d+\ell}} \\[1mm]
={}& \Big(\mfrac{d^2}{d^2-\ell^2}\Big)^{d-\ell+1/2} \Big(\mfrac{(\frac{n}{2}-d)^2}{(\frac{n}{2}-d)^2-\ell^2}\Big)^{(n+1)/2-d-\ell} \left(\mfrac{d}{d+\ell}\right)^{2\ell} \left(\mfrac{\frac{n}{2}-d}{\frac{n}{2}-d+\ell}\right)^{2\ell} \\[1mm]
={}& \Big(\mfrac{d}{d-\frac{1}{d}\ell^2}\Big)^{d-\ell+1/2} \Big(\mfrac{\frac{n}{2}-d}{\frac{n}{2}-d-\frac{1}{n/2-d}\ell^2}\Big)^{(n+1)/2-d-\ell} \left(\mfrac{\ell}{\ell+\frac{1}{d}\ell^2}\right)^{2\ell} \left(\mfrac{\ell}{\ell+\frac{1}{n/2-d}\ell^2}\right)^{2\ell} \\[1mm]
\sim{}& \exp\left(\tfrac{1}{d}\ell^2\right)\, \exp\left(\tfrac{1}{d(k-1)}\ell^2\right)\, \exp\left(-\tfrac{2}{d}\ell^2\right)\, \exp\left(-\tfrac{2}{d(k-1)}\ell^2\right) \\[1mm]
={}& \exp\left(-\tfrac{k}{d(k-1)}\ell^2\right)
\end{align*}
where the first line follows by applying \eqref{E:StirlingApprox} and simplifying, and the fourth line follows by applying limit identities, remembering that $n-2d=(2d+1)(k-1)$ and $\ell = O(\sqrt{d})$. So we have proved (b). In particular, note that $\e_\ell=\Theta(\e_0)$.

For any positive real constant $\kappa>0$, let $u_{\kappa}=\kappa\sqrt{d(k-1)/k}$ and note that $u_{\kappa}=\Theta(\sqrt{d})$. Now,
\begin{align}
  a(\lfloor u_\kappa \rfloor) &= \e_0+2\medop\sum_{\ell=1}^{\lfloor u_\kappa \rfloor}\e_\ell \nonumber\\
  &\sim \e_0\biggl(1+2\medop\sum_{\ell=1}^{\lfloor u_\kappa \rfloor}\exp\left(-\tfrac{k}{d(k-1)}\ell^2\right)\biggr) \nonumber\\
  &\sim 2\e_0\,\biggl(\medop\sum_{\ell=0}^{\lfloor u_\kappa \rfloor}\exp\left(-\tfrac{k}{d(k-1)}\ell^2\right)\biggr) \nonumber\\
  &\sim 2\e_0\,\int_{0}^{u_\kappa} \exp\left(-\tfrac{k}{d(k-1)}\ell^2\right)\,d\ell \nonumber\\
  &= \e_0\sqrt{\tfrac{k-1}{k}d\pi}\,\erf(\kappa) \nonumber\\
  &\sim \mbinom{n}{2d}\,\erf(\kappa)  \label{E:aExpr}
\end{align}
where in the second line we used part (b) of this lemma and the definition of the function $a$, in the fourth line we approximated the sum with an integral, in the fifth we changed the variable of integration to $t=\sqrt{k/d(k-1)}\ell$ and applied the definition of $u_\kappa$, and in the last we applied \eqref{E:StirlingApprox} to $\e_0$ and $\binom{n}{2d}$ and performed a routine calculation recalling $n \sim 2kd$. In the first three lines, recall that $u_{\kappa}=\Theta(\sqrt{d})$ and the terms of the sum are comparable and hence each term is insignificant compared to the whole. On the other hand,
\begin{align}
  b(\lfloor u_\kappa \rfloor) &= \mbinom{n}{2d+1}-2\medop\sum_{\ell=0}^{\lfloor u_\kappa \rfloor}\e_\ell^* \nonumber\\
  &\sim (k-1)\left(\mbinom{n}{2d} - 2\medop\sum_{\ell=0}^{\lfloor u_\kappa \rfloor}\e_\ell\right) \nonumber\\
  &\sim (k-1)\left(\mbinom{n}{2d} - a(\lfloor u_\kappa \rfloor)\right) \nonumber\\
  &\sim (k-1)\mbinom{n}{2d}(1-\erf(\kappa)) \label{E:bExpr}
\end{align}
where the first line follows from the definition of $b$, the second using part (a) of this lemma and the fact that $\binom{n}{2d+1}=(k-1)\binom{n}{2d}$ since $n-2d=(2d+1)(k-1)$, the third by the definition of $a$ and because any term of the sum is insignificant compared to the whole, and the last by \eqref{E:aExpr}. Let $\kappa_0=\erf^{-1}(\frac{1}{2})$. For any $\kappa < \kappa_0$, using \eqref{E:aExpr} and \eqref{E:bExpr}, we have $(k-1)a(\lfloor u_\kappa \rfloor) <b(\lfloor u_\kappa \rfloor)$ and hence $u \geq \lfloor u_\kappa \rfloor$ by the definition of $u$. Similarly, $(k-1)a(\lfloor u_\kappa \rfloor) > b(\lfloor u_\kappa \rfloor)$ and hence $u < \lfloor u_\kappa \rfloor$ for any $\kappa > \kappa_0$. This proves (c). Finally, we have $Q \sim a(u)\sim\frac{1}{2}\binom{n}{2d}$ using \eqref{E:aExpr} and the fact that $u \sim u_{\kappa_0}$ by part (c) of this lemma. Since $c=2d$, $r=k-1$ and $n-2d=(2d+1)(k-1)$ in this case, $\MMS(n,k)=\frac{1}{2}\binom{n}{2d}$ and the proof of (d) is complete.
\end{proof}

\begin{lemma}\label{L:IPachievesk-1}
Let $n$ and $k$ be integers such that $n \rightarrow \infty$ with $n \equiv k-1 \mod{2k}$, $k=o(n)$, and $k \geq 5$ is odd. Then the optimal value of $\mathcal{I}_{(n,k)}$ is $Q$.
\end{lemma}

\begin{proof}
Observe that $d \rightarrow \infty$ since $k=o(n)$. For $t \in \{1,\ldots,u\}$, let $\alpha_t$ be the slack in the constraint of $\mathcal{I}_{(n,k)}$ given by setting $\ell=t$ in \eqref{E:C1}, and let $\alpha_0$ be the slack in constraint \eqref{E:C2} of $\mathcal{I}_{(n,k)}$. Similarly, for $t \in \{u+1,\ldots,d\}$, let $\beta_{t}$ be the slack in the constraint of $\mathcal{I}_{(n,k)}$ given by setting $\ell=t$ in \eqref{E:C3}. Consider the candidate solution for $\mathcal{I}_{(n,k)}$ in which every variable is 0 except
\begin{itemize}
    \item
$x_{u+1+i,2u+1-i}=\e_{u-2i}$ for each $i \in \{0,\ldots,\lfloor\frac{u-1}{2}\rfloor\}$,
    \item
$x_{u+1+i,2u-i}=\e_{u-2i-1}$ for each $i \in \{0,\ldots,\lfloor\frac{u-2}{2}\rfloor\}$,
    \item
$x_{\lfloor3u/2\rfloor+1,\lfloor3u/2\rfloor+1}=\lfloor\frac{1}{2}\e_{0}\rfloor$.
\end{itemize}
We claim that this is indeed a feasible solution with objective value $Q$. Note that, for this assignment, $\alpha_{\ell}=0$ for each $\ell \in \{0,\ldots,u\}$. Because twice the sum of the constraints \eqref{E:C1} for $\ell \in \{1,\ldots,u\}$ and \eqref{E:C2} has the objective function of $\mathcal{I}_{(n,k)}$ as its left hand side and $Q=2\lfloor\frac{1}{2}a(u)\rfloor$ as its right hand side, this assignment has objective value $Q$. So it only remains to show that $\beta_{\ell} \geq 0$ for each $\ell \in \{u+1,\ldots,d\}$ and hence that the solution is feasible. For each $\ell \in \{2u+2,\ldots,d\}$, \eqref{E:C3} is clearly satisfied since its left hand side is 0. Furthermore, for each $\ell \in \{u+1,\ldots,2u+1\}$, it can be seen that at most two of the variables contributing to the left hand side are nonzero. In fact, we have
\begin{align}
\beta_{2u+1-i} &=\e^*_{2u+1-i}-\e_{u-2i+1}-\e_{u-2i} & \text{for each $i \in \{1,\ldots,\lfloor\tfrac{u-1}{2}\rfloor\}$} \label{e:beta1}\\
\beta_{u+1+i} &=\e^*_{u+1+i}-\e_{u-2i}-\e_{u-2i-1} & \text{for each $i \in \{0,\ldots,\lfloor\tfrac{u-2}{2}\rfloor\}$} & \label{e:beta2}\\
\beta_{2u+1} &=\e^*_{2u+1}-\e_{u} & \label{e:beta3}\\
\beta_{\lfloor3u/2\rfloor+1} &= \e^*_{\lfloor3u/2\rfloor+1}-\e_{1}-2\lfloor\tfrac{1}{2}\e_{0}\rfloor. & \label{e:beta4}
\end{align}

We now obtain two expressions which we will use to approximate, respectively, the positive and negative terms in the expressions on the right hand sides of \eqref{e:beta1}--\eqref{e:beta4}. For each $\ell \in \{0,\ldots,2u+1\}$, we have from Lemma~\ref{L:QAymptoticsk+1mod2k}(a) and (b) that
\begin{equation}\label{e:posExpr}
\e^*_\ell \sim  (k-1)\e_\ell  \sim (k-1)\e_0\exp\left(-\tfrac{k}{d(k-1)}\ell^2\right).
\end{equation}
Also, for each $\ell \in \{0,\ldots,u\}$, using Lemma~\ref{L:QAymptoticsk+1mod2k}(b) we obtain
\begin{equation}\label{e:negExpr}
\e_\ell+\e_{\ell+1} < 2\e_\ell \sim 2\e_0\exp\left(-\tfrac{k}{d(k-1)}\ell^2\right).
\end{equation}
Using \eqref{e:posExpr} and \eqref{e:negExpr} it is not too difficult to establish that the expressions in \eqref{e:beta1}--\eqref{e:beta4} are nonnegative in each case. We demonstrate how to do this in the case of \eqref{e:beta1}, where the expressions come closest to being negative.

Let $i$ be an arbitrary element of $\{1,\ldots,\lfloor\tfrac{u-1}{2}\rfloor\}$. Using \eqref{e:beta1}, \eqref{e:posExpr} and \eqref{e:negExpr} and then factorising, we have
\begin{align}
\beta_{2u+1-i} > \e^*_{2u+1-i}-2\e_{u-2i} &\sim (k-1)\e_{2u+1-i}-2\e_{u-2i} \nonumber\\
&\sim(k-1)\e_0\exp\left(-\tfrac{k(2u+1-i)^2}{d(k-1)}\right) - 2\e_0\exp\left(-\tfrac{k(u-2i)^2}{d(k-1)}\right) \nonumber\\
&=\e_0\exp\left(-\tfrac{k(2u+1-i)^2}{d(k-1)}\right) \left(k-1-2\exp\left(\tfrac{k(3u+1-3i)(u+1+i)}{d(k-1)}\right)\right). \label{e:factoredSlack}
\end{align}
The first exponential in \eqref{e:factoredSlack} approaches a positive constant since $u=\Theta(\sqrt{d})$ by Lemma~\ref{L:QAymptoticsk-1mod2k}(c), so to prove that $\beta_{2u+1-i} \geq 0$ it suffices to show that the second exponential approaches a constant strictly less that $\frac{k-1}{2}$. Now
\[\tfrac{k(3u+1-3i)(u+1+i)}{d(k-1)} \leq \tfrac{k(3u-2)(u+2)}{d(k-1)} \sim \tfrac{3ku^2}{d(k-1)} \sim 3\bigl(\erf^{-1}(\tfrac{1}{2})\bigr)^2\]
where the first inequality follows because $(3u+1-3i)(u+1+i)$ is maximised when $i=1$ given that $i \geq 1$, the first $\sim$ follows because $u=\Theta(\sqrt{d})$ and the second by Lemma~\ref{L:QAymptoticsk-1mod2k}(c). Thus, because $\erf^{-1}(\tfrac{1}{2})<0.477$, it is easy to calculate that the second exponential in \eqref{e:factoredSlack} approaches a constant less than $1.98$. Since $k \geq 5$, this is less than $\frac{k-1}{2}$. Thus $\beta_{2u+1-i} \geq 0$.
Very similar arguments show that the expressions in \eqref{e:beta2}--\eqref{e:beta4} are nonnegative in each case.
\end{proof}

\begin{proof}[\textbf{\textup{Proof of Theorem~\ref{t:k=3Asymptotics}(b).}}]
This follows immediately from Lemmas~\ref{L:rkIPBound}, \ref{L:QAymptoticsk-1mod2k}(d) and \ref{L:IPachievesk-1}.
\end{proof}

\section{Conclusion}

As a result of Theorem~\ref{t:k=3Asymptotics}(b), the asymptotics of $\SP(n,k)$ when $k = o(n)$ is odd and $n \equiv k-1 \mod{2k}$ are only unknown when $k=3$. We believe the statement is still true in this case.
\begin{conjecture}
Let $n$ be an integer such that $n \rightarrow \infty$ with $n \equiv 2 \mod{6}$. Then $\SP(n,3) \sim \MMS(n,3)$.
\end{conjecture}

In an attempt to find evidence supporting this conjecture, we observed that if we relax one of the integer programs $\mathcal{I}_{(n,k)}$ to a linear program $\mathcal{L}_{(n,k)}$, then the optimal value of $\mathcal{L}_{(n,k)}$ exceeds the optimal value of $\mathcal{I}_{(n,k)}$ by at most $2|\Phi|<2u(d-u)$. This is because a feasible solution for $\mathcal{I}_{(n,k)}$ can be obtained from an optimal solution for $\mathcal{L}_{(n,k)}$ by simply taking the floor of each variable. So clearly the two optimal values are asymptotic to each other as $n$ becomes large. Thus we implemented the linear relaxation $\mathcal{L}_{(n,k)}$ in the linear programming solver Gurobi \cite{gurobi}. We proceeded to solve $\mathcal{L}_{(n,3)}$ for all $n\equiv 2 \mod{6}$ where $26 \leq n \leq 18000$ (the program is trivial for $n<26$), and in all cases the optimal solution found gave an objective value that matched $Q$ to at least 11 significant figures. Since the objective value for $\mathcal{L}_{(n,3)}$ is asymptotic to the objective value of $\mathcal{I}_{(n,3)}$, and in view of Lemma~\ref{L:rkIPBound} and Lemma~\ref{L:QAymptoticsk-1mod2k}(d), these calculations support our belief that the conjecture holds.

In view of the work in this paper and \cite{Paper1} it has been established that $\SP(n,k)$ is asymptotic to $\MMS(n,k)$ in a very wide variety of cases, with the main unresolved situation being when $n$ is odd, $k=o(n)$ and $k-r$ is bounded. It seems likely that addressing these cases will require developing new constructions. Of course, it is also of interest to find exact values of $\SP(n,k)$. Table~\ref{table:exactValues} and \cite[Lemma 16]{Paper1} and  show that the upper bound on $\SP(n,k)$ given by Theorem~\ref{t:upperBound} is sometimes tight but we do not know how often this is the case.

\section*{Acknowledgments}\label{ackref}
Adam~Gowty received support from an Australian Government Research Training Program Scholarship. Daniel~Horsley received support from ARC grants DP150100506 and FT160100048.

\bibliographystyle{abbrvnat}
\bibliography{bibliography}

\end{document}